\documentclass[leqno,11pt]{amsart}

\usepackage{graphicx}
\usepackage{color}
\usepackage{epstopdf}
\usepackage{pdfsync}
\usepackage{amssymb,amsfonts,amsthm,amssymb}
\usepackage{amsmath}
\usepackage{mathrsfs}
\usepackage{mathtools}
\usepackage{calc}
\usepackage{verbatim}
\usepackage{color}
\usepackage{transparent} 
\usepackage{stmaryrd}
\usepackage[left=3.1cm,right=3.1cm,top=3.1cm,bottom=3.1cm]{geometry}
\newcommand*{\mailto}[1]{\href{mailto:#1}{\nolinkurl{#1}}}

\newcommand{\R}{{\mathbb{R}}}
\newcommand{\one}{\chi}
\newcommand*{\Lcdot}{\raisebox{-0.25ex}{\scalebox{1.1}{$\cdot$}}}

\newtheorem{theorem}{Theorem}[section]
\newtheorem{lemma}{Lemma}[section]
\newtheorem{proposition}{Proposition}[section]
\newtheorem{corollary}{Corollary}[section]
\newtheorem{remark}{Remark}[section]

\def\be#1{\begin{equation}\label{#1}}
\def\ee{\end{equation}}

\def\bea{\begin{eqnarray}}
\def\eea{\end{eqnarray}}
\def\bth#1{\begin{theorem}\label{#1}}
\def\eth{\end{theorem}}
\def\brem#1{\begin{remark}\label{#1}}
\def\erem{\end{remark}}
\def\blem#1{\begin{lemma}\label{#1}}
\def\elem{\end{lemma}}
\def\beas{\begin{eqnarray*}}
\def\eeas{\end{eqnarray*}}
\def\barr{\begin{array}}
\def\earr{\end{array}}
\def\bdm{\begin{displaymath}}
\def\edm{\end{displaymath}}
\def\bcor#1{\begin{corollary}\label{#1}}
\def\ecor{\end{corollary}}
\def\cW{\mathcal{W}}

\def\div{\mathrm{div}}

\def\nn{\nonumber}
\usepackage{amssymb}

\numberwithin{equation}{section}
\allowdisplaybreaks

\begin{document} 
\title[Shape optimization of a focusing acoustic lens]{\vspace{-2.0cm}Sensitivity analysis for shape optimization of a focusing acoustic lens in lithotripsy}

\author[V. Nikoli\' c]{Vanja Nikoli\' c}
\address{Insitute of Mathematics\\ University of Klagenfurt\\
Universit\"atsstrasse 65-57\\ 9020 Klagenfurt am W\"orthersee\\ Austria}
\email{vanja.nikolic@aau.at}

\author[B. Kaltenbacher]{Barbara Kaltenbacher}
\address{Insitute of Mathematics\\ University of Klagenfurt\\
Universit\"atsstrasse 65-57\\ 9020 Klagenfurt am W\"orthersee\\ Austria}
\email{barbara.kaltenbacher@aau.at}
\begin{abstract}  We are interested in shape sensitivity analysis for an optimization problem arising in medical applications of high intensity focused ultrasound. The goal is to find the optimal shape of a focusing acoustic lens so that the desired acoustic pressure at a kidney stone is achieved. Coupling of the silicone acoustic lens and nonlinearly acoustic fluid region is modeled by the Westervelt equation with nonlinear strong damping and piecewise constant coefficients. We follow the variational approach to calculating the shape derivative of the cost functional which does not require computing the shape derivative of the state variable; however assumptions of certain spatial regularity of the primal and the adjoint state are needed to obtain the derivative, in particular for its strong form according to the Delfour-Hadamard-Zol\' esio Structure Theorem.
\end{abstract}


\maketitle
\vspace{-5.5mm}
\section{Introduction}
 The present paper is concerned with the shape optimization problem arising in lithotripsy, where an optimal focusing of the ultrasound waves is needed in order to concentrate the ultrasound pressure at the kidney stone and avoid lesions of the surrounding tissue. \\
\indent The design of currently used devices in lithotripsy is mainly based on two different principles (see, e.g, \cite{KV}, \cite{manfred}, \cite{uffc2002}, \cite{Veljovic}): excitation and self-focusing by a piezo mozaic or an electro-magneto-mechanical principle. \vspace{0.7mm}\\
\begin{center}
   \def\svgwidth{207pt}
\begingroup%
  \makeatletter%
  \providecommand\color[2][]{%
    \errmessage{(Inkscape) Color is used for the text in Inkscape, but the package 'color.sty' is not loaded}%
    \renewcommand\color[2][]{}%
  }%
  \providecommand\transparent[1]{%
    \errmessage{(Inkscape) Transparency is used (non-zero) for the text in Inkscape, but the package 'transparent.sty' is not loaded}%
    \renewcommand\transparent[1]{}%
  }%
  \providecommand\rotatebox[2]{#2}%
  \ifx\svgwidth\undefined%
    \setlength{\unitlength}{968.42395242bp}%
    \ifx\svgscale\undefined%
      \relax%
    \else%
      \setlength{\unitlength}{\unitlength * \real{\svgscale}}%
    \fi%
  \else%
    \setlength{\unitlength}{\svgwidth}%
  \fi%
  \global\let\svgwidth\undefined%
  \global\let\svgscale\undefined%
  \makeatother%
  \begin{picture}(1,0.76260738)%
    \put(0,0){\includegraphics[width=\unitlength]{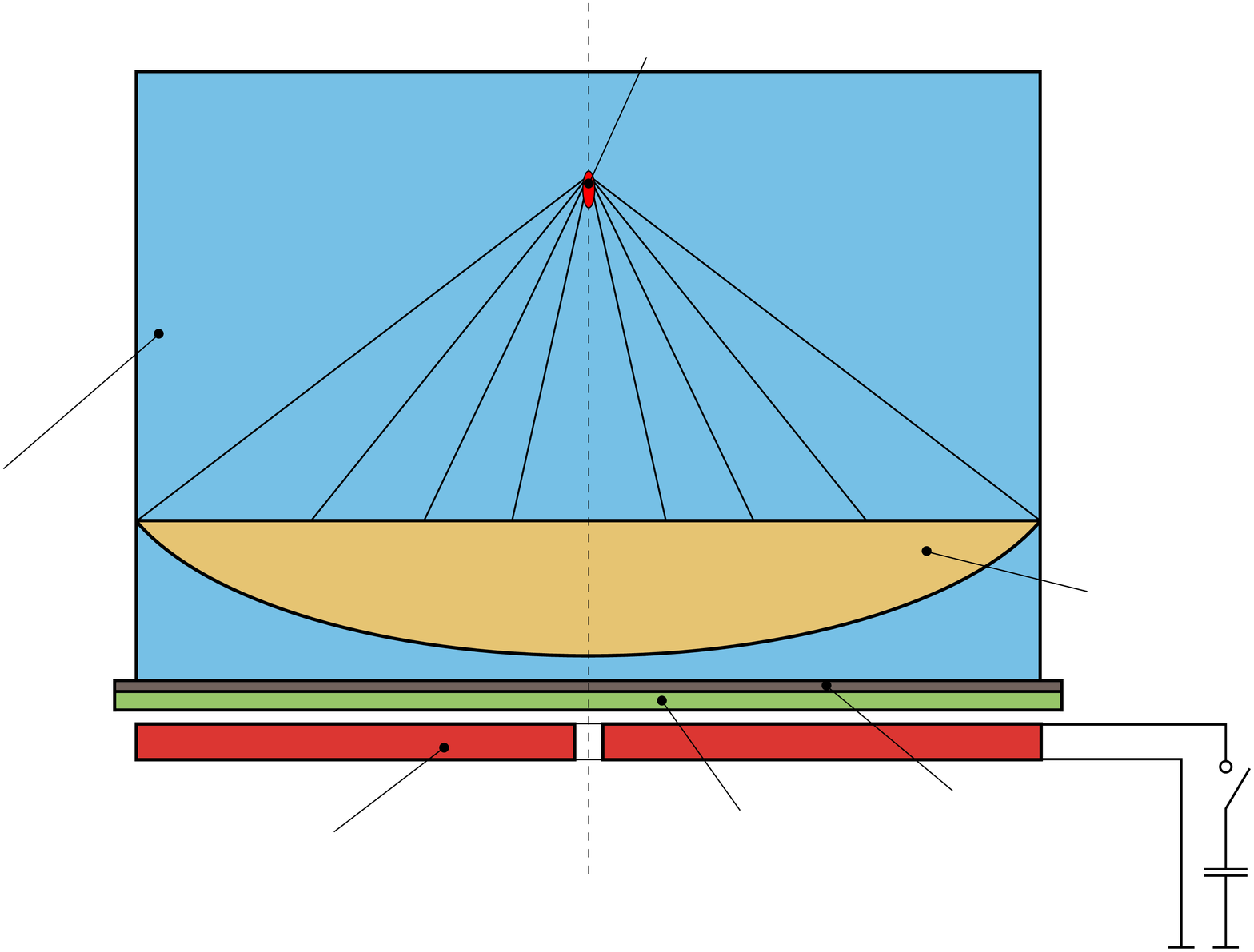}}%
    \put(0.18847806,0.14273979){\color[rgb]{0,0,0}\makebox(0,0)[lt]{\begin{minipage}{0.10739098\unitlength}\raggedright \end{minipage}}}%
    \put(0.2209543,0.07883277){\color[rgb]{0,0,0}\makebox(0,0)[lt]{\begin{minipage}{0.2100615\unitlength}\raggedright \small Coil\end{minipage}}}%
    \put(0.50207329,0.08638679){\color[rgb]{0,0,0}\makebox(0,0)[lt]{\begin{minipage}{0.24074464\unitlength}\raggedright \small Membrane\end{minipage}}}%
    \put(0.72648888,0.1139189){\color[rgb]{0,0,0}\makebox(0,0)[lt]{\begin{minipage}{0.14397472\unitlength}\raggedright \small Rubber\end{minipage}}}%
    \put(0.50080019,0.73600054){\color[rgb]{0,0,0}\makebox(0,0)[lb]{\smash{\small Kidney stone}}}%
    \put(0.28536583,0.79979627){\color[rgb]{0,0,0}\makebox(0,0)[lb]{\smash{\small Axis of Rotation}}}%
    \put(0.87764683,0.26170114){\color[rgb]{0,0,0}\makebox(0,0)[lb]{\smash{\small Lens} $\Omega_+$}}%
    \put(-0.11029899,0.31652083){\color[rgb]{0,0,0}\makebox(0,0)[lb]{\smash{\small Fluid} $\Omega_-$}}%
  \end{picture}%
\endgroup%

 {\scriptsize Schematic of a power source in lithotripsy \\ based on the electromagnetic principle}
\end{center}
We investigate the latter case, where Lorentz forces acting on a membrane radiate an acoustic pulse in a fluid; the pulse is then focused by a silicone lens at a kidney stone. The aim of obtaining a sharp focus of the acoustic pressure exactly at the desired location leads to the task of optimizing the shape of the acoustic lens. We will consider the case where the lens is modeled as an acoustic medium surrounded by a nonlinearly acoustic fluid.\\
\indent One of the commonly used models for propagation of nonlinear ultrasound is the Westervelt equation
\begin{align} \label{westervelt}
(1-2ku)\ddot{u} -c^2 \Delta u - b\Delta \dot{u} = 2k(\dot{u})^2,
\end{align}
given here in terms of the acoustic pressure $u$, where $b$ denotes the diffusivity and $c$ the speed of sound, $k = \beta_a / \lambda$, $\lambda=\varrho c^2$ is the bulk modulus, $\varrho$ is the mass density, $\beta_a = 1 + B/(2A)$, and $B/A$ represents the parameter of nonlinearity. For a detailed derivation of \eqref{westervelt} we refer to \cite{HamiltonBlackstock},  \cite{manfred} and \cite{Westervelt}. \\
\indent Westervelt's equation is a quasilinear wave equation which can degenerate due to the factor $1-2ku$. To avoid this degeneracy, i.e. to find an essential bound for the acoustic pressure $u$, Sobolev embedding $H^2(\Omega) \rightarrow L^\infty(\Omega)$ is employed (cf. \cite{KL08}, \cite{KLV10}) and thus for this model only very smooth solutions can be shown to exist. However, $H^2$ regularity on the whole domain is too high of a demand in the case of interface coupling of acoustic regions with different material parameters. \\
\indent In \cite{BKR13}, a nonlinear damping term was introduced to the Westervelt equation
\begin{align} \label{westervelt_damp}
(1-2ku)\ddot{u}-c^2\Delta u -\div (b((1-\delta) +\delta|\nabla \dot{u}|^{q-1})\nabla \dot{u})
=2k(\dot{u})^2,
\end{align} 
with $\delta \in (0,1)$, $q \geq 1$, $q>d-1$, $d \in \{1,2,3\}$, which allowed to show existence of weak solutions with $W^{1,q+1}$ regularity in space, and in turn well-posedness of the acoustic-acoustic coupling problem. The case of the acoustic-acoustic coupling, which we are interested in, is modeled by the presence of spatially varying coefficients in the weak form of the  equation \eqref{westervelt_damp} (see \cite{BGT97} for the linear and \cite{BKR13} and \cite{VN} for the nonlinear case) as follows:
\begin{equation}\label{ModWest}
\begin{cases} 
\text{Find} \ u  \ \text{such that} \vspace{1.5mm}\\
\int_0^T \int_{\Omega} \{\frac{1}{\lambda(x)}(1-2k(x)u)\ddot{u} \phi+\frac{1}{\varrho(x)}\nabla u \cdot \nabla \phi + b(x)(1-\delta(x))\nabla \dot{u} \cdot \nabla \phi \vspace{1.5mm}\\
\quad \quad \quad +b (x)\delta(x)|\nabla \dot{u}|^{q-1} \nabla \dot{u} \cdot \nabla \phi-\frac{2k(x)}{\lambda(x)}(\dot{u})^2 \phi\} \, dx \, ds =0 \vspace{1.5mm}\\
\text{holds for all test functions} \ \phi \in \tilde{X},
\end{cases}
\end{equation}
with $(u,\dot{u})\vert_{t=0}=(u_0,u_1)$, and appropriately chosen test space $\tilde{X}$. In this model $b$ denotes the quotient between the diffusivity and the bulk modulus, while the other coefficients retain their meaning. For brevity, we emphasized the space dependence of coefficients in \eqref{ModWest}, while omitting space and time dependence of $u$ and the test function in the notation.  \vspace{-1mm}
\subsection{Some difficulties related to the model \eqref{ModWest}.} In \cite{KP}, the problem of optimizing the excitation part of the boundary in lithotripsy is considered, with the initial-boundary value problem for the Westervelt equation \eqref{westervelt} as the optimization constraint. This situation arises when excitation and self-focusing of high intensity ultrasound is performed by a piezo-mosaic. Compared to the the problem investigated there, the case of focusing by an acoustic lens presents us with several additional challenges. Not only a part of the domain boundary is optimized, but a subdomain which lies in the interior of the domain; this implies providing shape sensitivity analysis for an acoustic-acoustic interface problem. Insufficient spatial regularity of the primal (at most $W^{1,q+1}$ in space) and the adjoint state ($H^{1}$ in space) on the whole domain does not allow for the shape derivative to be expressed in terms of the boundary integrals. However, it turns out that the state variable exhibits $H^2$-regularity on each of the subdomains, provided that its gradient remains essentially bounded in space and time on the whole domain, which we will be able to use in our advantage. \\ \indent Working with the state equation also implies handling the nonlinear damping term of the (time-dependent) $q$-Laplace type, which is in itself a nontrivial task. Results on hyperbolic equations involving a $q$-Laplace damping term are sparse, and to the authors' best knowledge do not include considerations of interface or shape optimization problems.  \vspace{-1mm}
\subsection{Notational remark}  To avoid confusion, it should be emphisized that we use a dot notation for time differentiation and $t \in [0,T]$ for the physical time variable, while $\tau \in \mathbb{R}$ will be used for the artificial time variable to indicate varying domains. If $\Omega_+$ is the initial shape of the lens, then  $\Omega_{+,\tau}$ will denote the perturbed lens obtained by moving points into the direction of some vector field $h$ by some steplength $\tau$. 
\subsection{Outline of the paper} The rest of the paper is organized as follows. We formulate the shape optimization problem in Section \ref{sho_problem} and perform an analysis of the state and the adjoint problem. Section \ref{existence_opt_shapes} deals with the question of existence of minimizers. In Section \ref{state_perturbed} we investigate the state equation on the domain with the perturbed lens. Some preliminary results needed to obtain the shape derivative of the cost functional are given in Section \ref{preliminary}. Finally, Section \ref{computation_shape} contains computation of the shape derivative.
\section{Shape optimization problem} \label{sho_problem}
\noindent Let $\Omega \subset \mathbb{R}^d$, $d \in \{1,2,3\}$, be a fixed bounded domain with strongly Lipschitz  boundary $\partial \Omega$, and $\Omega_{+}$ a subdomain, representing the lens, such that $\overline{\Omega}_{+} \subset \Omega$ and $\Omega_{+}$  has strongly Lipschitz boundary $\partial \Omega_+=\Gamma$. \\

\begin{minipage}{0.4\textwidth}
 \def\svgwidth{120pt}
\hspace{5mm}  
\begingroup%
  \makeatletter%
  \providecommand\color[2][]{%
    \errmessage{(Inkscape) Color is used for the text in Inkscape, but the package 'color.sty' is not loaded}%
    \renewcommand\color[2][]{}%
  }%
  \providecommand\transparent[1]{%
    \errmessage{(Inkscape) Transparency is used (non-zero) for the text in Inkscape, but the package 'transparent.sty' is not loaded}%
    \renewcommand\transparent[1]{}%
  }%
  \providecommand\rotatebox[2]{#2}%
  \ifx\svgwidth\undefined%
    \setlength{\unitlength}{321.05312653bp}%
    \ifx\svgscale\undefined%
      \relax%
    \else%
      \setlength{\unitlength}{\unitlength * \real{\svgscale}}%
    \fi%
  \else%
    \setlength{\unitlength}{\svgwidth}%
  \fi%
  \global\let\svgwidth\undefined%
  \global\let\svgscale\undefined%
  \makeatother%
  \begin{picture}(1,0.88560109)%
    \put(0,0){\includegraphics[width=\unitlength]{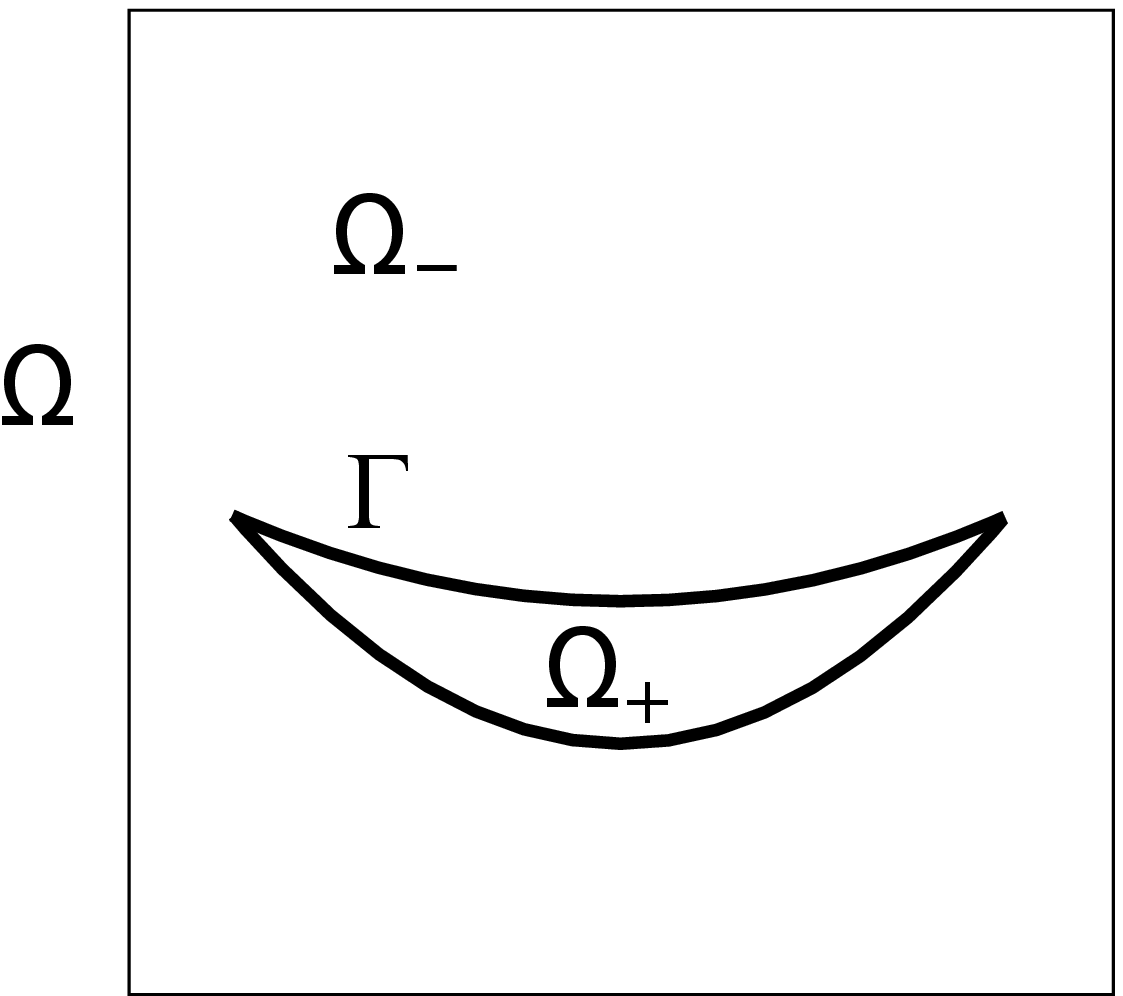}}%
    \put(0.48408068,0.25951276){\color[rgb]{0,0,0}\makebox(0,0)[lb]{\smash{Ω}}}%
    \put(0.56155101,0.24019995){\color[rgb]{0,0,0}\makebox(0,0)[lb]{\smash{+}}}%
    \put(0.2934314,0.64808884){\color[rgb]{0,0,0}\makebox(0,0)[lb]{\smash{Ω}}}%
    \put(0.36737739,0.63863532){\color[rgb]{0,0,0}\makebox(0,0)[lb]{\smash{-}}}%
    \put(0.31140779,0.41982592){\color[rgb]{0,0,0}\makebox(0,0)[lb]{\smash{Γ}}}%
    \put(-0.00472079,0.51225733){\color[rgb]{0,0,0}\makebox(0,0)[lb]{\smash{Ω}}}%
  \end{picture}%
\endgroup%

\\ \hspace*{5mm}  {\scriptsize Lens $\Omega_+$ and fluid $\Omega_-$ regions}
\end{minipage}
\begin{minipage}{0.55\textwidth}
\vspace{-5.5mm}
 We denote by $\Omega_{-}=\Omega \setminus \overline{\Omega}_{+}$ the part of the domain representing the fluid region. We then have  $\partial \Omega_-=\Gamma \cup \partial \Omega$. \\
$n_{+}$, $n_{-}$  will stand for the unit outer normals to lens $\Omega_{+}$ and fluid region $\Omega_{-}$. 
Restrictions of a function $v$ to $\Omega_{+,-}$ will be denoted by $v_{+}$, $v  _{-}$ and $\llbracket v \rrbracket:=v_+-v_-$ will denote the jump over $\Gamma$. \\
Note that the assumptions on the regularity of the subdomains will eventually have to be strengthened to $C^{1,1}$ in order to express the shape derivative in terms of the boundary integrals over the interface $\Gamma$.
\end{minipage} \vspace{1.5mm}\\
\noindent We consider the following optimization problem
\begin{align} \label{shape_opt_problem}
\begin{cases}
\displaystyle \min_{\substack{\Omega_+ \in \mathcal{O}_{\text{ad}}\\ u \in L^2(\Omega \times [0,T])}} J(u, \Omega_+)\equiv \min_{\substack{\Omega_+ \in \mathcal{O}_{\text{ad}}\\ u \in L^2(\Omega \times [0,T])}} \, \int_0 ^T \int_{\Omega} (u-u_{\text{d}})^2 \, dx \, ds \vspace{2mm} \\
 \text{subject to the constraint (1.3)},
\end{cases}
\end{align}
with the test space $\tilde{X}=L^2(0,T;W_0^{1,q+1}(\Omega))$, and where the coefficients $\lambda$, $k$, $\varrho$, $b$ and $\delta$  are piecewise continuously differentiable and allowed to jump only over the interface $\partial \Omega_+=\Gamma$, i.e.  
\begin{align} \label{coeff_1}
\begin{cases}
\lambda, k, \varrho, b, \delta \in L^{\infty}(\Omega), \\
w_i:=w\vert_{\Omega_i} \in C^1(\Omega_i), \ \text{for} \ w \in \{b, \varrho, \lambda, \delta, k\}, i \in\{+,-\}.
\end{cases}
\end{align}
$u_d \in L^2(0,T;L^2(\Omega))$ is the desired acoustic pressure; $\mathcal{O}_{\text{ad}}$ represents the set of admissible domains and is defined as follows:
\begin{align*}
\mathcal{O}_{\text{ad}}=\{\Omega_+:\overline{\Omega}_+ \subset \Omega, \ \Omega_+ \ \text{is open and Lipschitz with uniform Lipschitz constant} \ L_{\mathcal{O}}\}.
\end{align*} 
We assume that $q \geq 1$, $q>d-1$ for the state and the adjoint problem to be well-posed; however we will need to strengthen this assumption later to $q>2$ in order to prove certain properties needed for the characterization of the shape derivative. \\
The strong form of the PDE constraint is given by
\begin{align} \label{state_problem}
\begin{cases}
\frac{1}{\lambda(x)}(1-2k(x)u)\ddot{u}-\div(\frac{1}{\varrho(x)}\nabla u)
-\text{div}(b(x)((1-\delta(x)) +\delta(x)|\nabla \dot{u}|^{q-1})\nabla \dot{u}) \vspace{1.5mm} \\
=\frac{2k(x)}{\lambda(x)}(\dot{u})^2 \quad \text{ in } \Omega_{+} \cup \Omega_{-}, \vspace{1.5mm} \\
 \llbracket u \rrbracket=0 \quad \text{on} \ \Gamma=\partial \Omega_+, \vspace{1.5mm} \\
  \Bigl \llbracket \frac{1}{\varrho(x)} \frac{\partial u}{\partial n_{+}}+b(x)(1-\delta(x))\frac{\partial \dot{u}}{\partial n_{+}}+b(x)\delta(x)|\nabla \dot{u}|^{q-1}\frac{\partial \dot{u}}{\partial n_{+}}\Bigr \rrbracket=0 \quad \text{on} \ \Gamma=\partial \Omega_+, \vspace{1.5mm} \\
 u=0 \quad \text{on} \ \partial \Omega, \vspace{1.5mm}\\
(u,\dot{u})\vert_{t=0}=(u_0,u_1).
\end{cases}
\end{align}
For simplicity of exposition we consider homogeneous Dirichlet boundary conditions on the outer boundary $\partial\Omega$. We mention in passing that Neumann and absorbing boundary conditions as in \cite{KP}, \cite{VN} could be easily incorporated here as well, with some additional terms involving the Neumann boundary excitation in the analysis of the state equation, but with no changes in the shape derivative itself, since the outer boundary is not subject to modifications.
\subsection{Analysis of the state equation} \label{state}
Let us now add assumptions regarding the sign of the coefficients in the state problem:
\begin{align}\label{coeff_2}
\begin{cases}
\ \lambda, k, \varrho, b, \delta \in L^\infty(\Omega),\\
w_i:=w\vert_{\Omega_i} \in C^1(\Omega_i) \ \text{for} \ w \in \{b, \varrho, \lambda, \delta, k\}, \ i \in\{+,-\}, \\
\overline{w}:=|w|_{L^{\infty}(\Omega)} \ \text{for} \ w \in \{b, \varrho, \lambda, \delta, k\}, \ \overline{\delta}<1, \\
\ \exists \underline{\varrho}, \underline{b}, \underline{\delta}: \ \lambda \geq \underline{\lambda}>0, \ \varrho\geq\underline{\varrho}>0, \ b\geq\underline{b}>0, \ \delta \geq \underline{\delta}>0.
\end{cases}
\end{align} 
We will utilize the following well-posedness result (cf. Theorem 2.3. and Corollary 4.1, \cite{BKR13}):
\begin{proposition} (Local well-posedness of the state problem)\label{aa_corollary}
Let $q>d-1$, $q \geq 1$ and the assumptions \eqref{coeff_2} hold. For any $T>0$ there is a $\kappa_T>0$ such that for all  $u_0, u_1 \in W_0^{1,q+1}(\Omega)$ with
\begin{align*}
|u_1|^2_{L^2(\Omega)}+|\nabla u_0|^2_{L^2(\Omega)}+|\nabla u_1|^2_{L^2(\Omega)}+|\nabla u_0|^2_{L^{q+1}(\Omega)}+|\nabla u_1|^{q+1}_{L^{q+1}(\Omega)} \leq \kappa_T^2
\end{align*}
there exists a unique solution $u\in \cW \subset X=H^2(0,T;L^2(\Omega))\cap W^{1,\infty}(0,T;W_0^{1,q+1}(\Omega))$ of \eqref{ModWest}, and
\begin{eqnarray} \label{defcW}
\cW =\{v\in X
&:& \|\ddot{v}\|_{L^2(0,T;L^2(\Omega))}\leq \bar{m}  
 \wedge \|\nabla \dot{v}\|_{L^\infty(0,T;L^2(\Omega))}\leq \bar{m} \\
&& \wedge \|\nabla \dot{v}\|_{L^{q+1}(0,T;L^{q+1}(\Omega))}\leq \bar{M}\},\nonumber 
\end{eqnarray}
with \begin{equation*}
2\overline{k}C_{W_0^{1,q+1},L^\infty}^\Omega(\kappa_T+
T^{\frac{q}{q+1}} \bar{M}) <1, 
\end{equation*}  and $\bar{m}$ sufficiently small. Furthermore, the following estimate holds:
\begin{align} \label{ModWest_energyest}
&\|u\|^2_{L^{\infty}(0,T;L^{\infty}(\Omega))}+\|\ddot{u}\|^2_{L^2(0,T;L^2(\Omega))}+\|\nabla \dot{u}\|^2_{L^2(0,T;L^2(\Omega))}+\|\nabla \dot{u}\|^{q+1}_{L^{q+1}(0,T;L^{q+1}(\Omega))} \nonumber\\
&
+\|\dot{u}\|^2_{L^{\infty}(0,T;L^2(\Omega))}+\|\nabla u\|^2_{L^{\infty}(0,T;L^2(\Omega))}
+\|\nabla \dot{u}\|^2_{L^{\infty}(0,T;L^2(\Omega))}+\|\nabla \dot{u}\|^{q+1}_{L^{\infty}(0,T;L^{q+1}(\Omega))}\\
\leq&\, C(|u_1|^2_{L^2(\Omega)}+|\nabla u_0|^2_{L^2(\Omega)}+|\nabla u_1|^2_{L^2(\Omega)}+|\nabla u_1|^{q+1}_{L^{q+1}(\Omega)}), \nonumber
\end{align}
where $C$ depends on $\lambda,\varrho, b,\delta, k$, and the norm $C^{\Omega}_{H_0^1,L^4}$ of the embedding $H_0^1(\Omega) \hookrightarrow L^4(\Omega)$. 
\end{proposition}
Since $u-u_d \in L^2(0,T;L^2(\Omega))$, the cost functional is well-defined.  \\

\indent Degeneracy of the Westervelt equation is avoided here by employing the embedding $W_0^{1,q+1}(\Omega) \hookrightarrow L^{\infty}(\Omega)$ (since $q+1>d$), and the following estimate 
\begin{align*}
|u(t)|_{L^{\infty}(\Omega)}\leq&\,  C^{\Omega}_{W_0^{1,q+1},L^{\infty}}|\nabla u(t)|_{L^{q+1}(\Omega)}\\
\leq&\, C^{\Omega}_{W_0^{1,q+1},L^{\infty}}|\nabla u_0+\int_0^t \nabla \dot{u}(s) \, ds \,|_{L^{q+1}(\Omega)}\\
\leq&\, C^{\Omega}_{W_0^{1,q+1},L^{\infty}}\Bigl(|\nabla u_0|_{L^{q+1}(\Omega)}+\Bigl(t^q \int_0^t \int_\Omega |\nabla \dot{u}(y,s)|^{q+1}\, dy \, ds \Bigr)^{\frac{1}{q+1}}\Bigr),
\end{align*}
which leads to the bound
\begin{equation} \label{avoid_degeneracy}
\begin{aligned}
&1-a_0< 1-2ku < 1+a_0, \\
& a_0:= 2\overline{k}C^{\Omega}_{W_0^{1,q+1},L^{\infty}}(|\nabla u_0|_{L^{q+1}(\Omega)}+T^{\frac{q}{q+1}}\|\nabla \dot{u}\|_{L^{q+1}(0,T;L^{q+1}(\Omega))}). 
\end{aligned}
\end{equation}
Due to the Sobolev embedding $W^{1,q+1}(\Omega) \hookrightarrow C^{0,1-\frac{d}{q+1}}(\overline{\Omega})$, we know that $u$ is H\" older continuous in space, i.e. $u \in C^{0,1}(0,T;C^{0,1-\frac{d}{q+1}}(\overline{\Omega}))$.
\subsection{Inequalities} Before proceeding further, let us recall several useful inequalities which will help us deal with the nonlinear damping term appearing in the state equation. From now on, we will always use $C_q$ to denote a generic positive constant depending only on $q$.\\
At several instances, we will utilize the following representation formula for vectors $x,y \in \mathbb{R}^d$ (cf. Chapter 10, \cite{lindqvist}):
\begin{equation} \label{aa_formula}
\begin{aligned}
&|x|^{q-1}x-|y|^{q-1}y \\ 
 =& \, (x-y)\int_0^1|y+\sigma(x-y)|^{q-1} \, d\sigma
+(q-1)\int_0^1 \mathcal{L}(y+\sigma(x-y),(x-y) \, d\sigma,
\end{aligned}
\end{equation}
where 
\begin{equation}\label{calL}
\mathcal{L}(x,y)=|x|^{q-3}(x \cdot y) x
\end{equation}
as well as the inequality
\begin{align} \label{aa_ineq2}
|y+\sigma (x-y)|^{q-1} \leq |y|^{q-1}+|x|^{q-1}, \ \sigma \in [0,1].
\end{align}
For any $\eta \geq 0$, $q >0$ and $x,y \in \mathbb{R}^d$, the following inequality holds (cf. Lemma 5.1, \cite{LiuYan}):
\begin{align} \label{aa_ineq3}
||x|^{q-1}x-|y|^{q-1}y| \leq C_q \, |x-y|^{1-\eta}(|y|^{q-1+\eta}+|x|^{q-1+\eta}).
\end{align}
From here we also get 
\begin{align} \label{aa_ineq4_b}
||x|^{q-1}-|y|^{q-1}| \leq C_q \, |x-y|^{1-\eta}(|y|^{q-1+\eta}+|x|^{q-1+\eta}),
\end{align}
for $0 \leq \eta \leq 1$, $q>1$. \\
We also recall that the following estimate holds for $q \geq 1$:
\begin{align}  \label{aa_ineq5}
(|x|^{q-1}x-|y|^{q-1}y) \cdot (x-y) \geq 2^{1-q}|x-y|^{q+1} \geq 0.
\end{align} 
With the notation \eqref{calL}, as a simple consequence of \eqref{aa_ineq3}, we have for vectors $x,y,z,w$ and any $\eta \geq 0$, $q > 2$ that
\begin{equation} \label{aa_ineq6}
\begin{aligned}
& |\mathcal{L}(x,y)-\mathcal{L}(z,w)| \\
=&\, |(x \cdot y)(|x|^{q-3}x-|z|^{q-3}z)+|z|^{q-3}((y-w)\cdot x+w\cdot(x-z))z| \\
\leq& \, C_q|x-z|^{1-\eta}(|x|^{q-3+\eta}+|z|^{q-3+\eta})|x||y|+|z|^{q-2}(|y-w||x|+|w||x-z|).
\end{aligned}
\end{equation}  
\noindent We will also need Young's inequality (see for instance Appendix B, \cite{evans}) in the form
\begin{align} \label{Young}
\quad  \quad \quad |xy| \leq \varepsilon |x|^r+C(\varepsilon, r) |y|^{\frac{r}{r-1}}  \quad (\varepsilon >0, \ 1<r< \infty),
\end{align}
with $C(\varepsilon, r)=(r-1)r^{\frac{r}{r-1}}\varepsilon ^{-\frac{1}{1-r}}$. 
\subsection{Analysis of the adjoint problem} \label{adjoint}
\indent Since the state equation contains a $q$-Laplace type damping term, its linearization provides a significant challenge. The main difficulty in considering the linearized $q$-Laplacian lies in the need for an essential bound on the gradient of the solution $u$; this is necessary for the linearized operator to be bounded. \\
\indent We will consider the adjoint problem with the assumption that the solution $u$ of \eqref{ModWest} exhibits the following regularity: \\[1ex]
($\mathcal{H}_1$) \hspace{2mm} $ u \in H^2(0,T;W^{1,\infty}(\Omega)) \cap X$, \\[1ex]
\noindent Note that the hypothesis is equivalent to assuming Lipschitz continuity of the acoustic pressure in space, i.e. $u \in H^2(0,T;C^{0,1}(\overline{\Omega}))\cap X$ (cf. Chapter 2, Section 2.6.4, \cite{DZ}). Due to the embedding $H^2(0,T) \hookrightarrow W^{1,\infty}(0,T)$, if the hypothesis holds, we also know that $u \in W^{1,\infty}(0,T;W^{1,\infty}(\Omega))$.  \\
\indent The adjoint problem in weak form is then given as
\begin{align} \label{adjoint_weakform}
\begin{cases}
\int_0^T \int_{\Omega} \{\frac{1}{\lambda}(1-2ku)\ddot{p} \zeta+\frac{1}{\varrho}\nabla p \cdot \nabla \zeta-b(1-\delta)\nabla \dot{p} \cdot \nabla \zeta\vspace{1.5mm} \\
 \quad  -b \delta  (G_u(\nabla p))^{\Lcdot} \cdot \nabla \zeta\} \, dx \, ds = \int_0^T \int_{\Omega}  j^{\prime}(u) \zeta \, dx \, ds  \vspace{1.5mm} \\
\text{holds for all test functions} \ \zeta \in X^\prime,
\end{cases}
\end{align}
with $(p,\dot{p})\vert_{t=T}=(0,0)$, where $j(u)=(u-u_d)^2$, $X^\prime=L^2(0,T;H^1(\Omega))$ and we have used the notation
\begin{align*}
G_u(Y):=|\nabla \dot{u}|^{q-1}Y+(q-1)|\nabla \dot{u}|^{q-3}(\nabla \dot{u} \cdot Y) \nabla \dot{u}
\end{align*}
for the linearized $q$-Laplace operator. The strong formulation of the adjoint problem then reads as:
\begin{align} \label{adjoint_strong}
\begin{cases}
\frac{1}{\lambda}(1-2ku)\ddot{p}-\text{div}(\frac{1}{\varrho}\nabla p) 
 +\text{div}(b(1-\delta)\nabla \dot{p}) + \text{div}(b \delta  (G_u(\nabla p))^{\Lcdot}) \vspace{1.5mm} \\
=  j^{\prime}(u) \ \text{ in } \Omega_{+} \cup \Omega_{-}, \vspace{1.5mm} \\
 \llbracket p \rrbracket=0 \quad \text{on} \ \Gamma, \vspace{1.5mm} \\
  \Bigl \llbracket \frac{1}{\varrho} \frac{\partial p}{\partial n_{+}}-b(1-\delta)\frac{\partial \dot{p}}{\partial n_{+}}-b\delta \,(G_u(\nabla p) \cdot n_+)^{\Lcdot}\Bigr \rrbracket=0 \quad \text{on} \ \Gamma, \vspace{1mm} \\
p=0 \quad \text{on} \ \partial \Omega, \vspace{1.5mm} \\
(p,\dot{p})\vert_{t=T}=(0,0).
\end{cases}
\end{align}
To obtain a forward problem, let $t \in [0,T]$ and $\tilde{p}(t):=p(T-t)$ (cf. Lemma 3.17, \cite{troeltzsch}); then $(\tilde{p},\dot{\tilde{p}})\vert_{t=0}=(0,0)$, and we have the following problem for $\tilde{p}$: \\
\begin{align} \label{adjoint_0} 
\begin{cases}
\int_0^T \int_{\Omega} \Bigl\{\frac{1}{\lambda}(1-2ku)\ddot{\tilde{p} } \zeta+\frac{1}{\varrho}\nabla \tilde{p}  \cdot \nabla \zeta+b(1-\delta)\nabla \dot{\tilde{p} } \cdot \nabla \zeta \vspace{1.5mm} \\
 \quad  +b \delta  (|\nabla \dot{u}|^{q-1} \nabla \dot{\tilde{p} } 
 -  (|\nabla \dot{u}|^{q-1})^{\Lcdot} \nabla \tilde{p})  \cdot \nabla \zeta \vspace{1.5mm} \\
\quad-(q-1)b\delta \Bigl( (|\nabla \dot{u}|^{q-3})^{\Lcdot} (\nabla \tilde{p}  \cdot \nabla \dot{u})\nabla \dot{u} -  |\nabla \dot{u}|^{q-3}(\nabla \dot{\tilde{p}}\cdot \nabla \dot{u})\nabla \dot{u}  \vspace{1.5mm} \\
\quad+  |\nabla \dot{u}|^{q-3}(\nabla \tilde{p} \cdot \nabla \ddot{u})\nabla \dot{u} +  |\nabla \dot{u}|^{q-3}(\nabla \tilde{p} \cdot \nabla \dot{u})\nabla \ddot{u} \Bigr)\cdot \nabla \zeta \Bigr\} \, dx \, ds \vspace{1.5mm} \\
= \int_0^T \int_{\Omega}  j^{\prime}(u) \zeta \, dx \, ds,  \vspace{1.5mm}  \\
\text{holds for all test functions} \ \zeta \in X^\prime.
\end{cases}
\end{align}
\begin{proposition}\label{prop:adj}(Local well-posedness of the adjoint problem)
Let $q \geq 1$, $q >d-1$, assumptions \eqref{coeff_2} on coefficients and hypothesis \text{\normalfont($\mathcal{H}_1$)} hold. For sufficiently small $\bar{m}$, final time $T>0$, and $\|\nabla \dot{u}\|^{q-2}_{L^\infty(0,T;L^\infty(\Omega))}\|\nabla \ddot{u}\|_{L^\infty(0,T;L^\infty(\Omega))}$, there exists a unique weak solution $p \in \hat{X}=C^1(0,T;L^2(\Omega)) \cap H^1(0,T;H^1_0(\Omega))$ of \eqref{adjoint_strong}. 
\end{proposition}
\begin{proof}
\noindent The well-posedness of the adjoint problem can be obtained through standard Galerking approximation in space, energy estimates and weak limits (cf. Section 7.2, \cite{evans}). We will focus here on obtaining the crucial energy estimate. Testing \eqref{adjoint_0} with $\zeta(\sigma)=\dot{\tilde{p}} (\sigma)\one_{[0,t)}$ (first in a discretized setting and then via weak limit transfered to the continuous one) yields  \begin{align} \label{adjoint_estimate}
&\frac{1}{2}\Bigl[\int_{\Omega} (1-2ku(\sigma))(\dot{\tilde{p}}(\sigma))^2 \, dx  + \int_{\Omega} \frac{1}{\varrho}|\nabla \tilde{p}(\sigma)|^2 \, dx\Bigr]_0^t 
 +\int_0^t \int_{\Omega} b(1-\delta)|\nabla \dot{\tilde{p}}|^2 \, dx \, ds \nn \\
\leq& - \int_0^t \int_{\Omega} \frac{k}{\lambda}\dot{u} (\dot{\tilde{p}})^2 \, dx \, ds + \int_0^t \int_{\Omega} j^{\prime}(u) \dot{\tilde{p}} \, dx \, ds 
+\int_0^t \int_{\Omega} b \delta\varphi| \nabla \tilde{p}|  |\nabla \dot{\tilde{p}}| \, dx \, ds, 
\end{align}
where 
\begin{align*}
\varphi:=& \, |(|\nabla \dot{u}|^{q-1})^{\Lcdot}|+(q-1)|(|\nabla \dot{u}|^{q-3})^{\Lcdot}||\nabla \dot{u}|^2 
+2(q-1)|\nabla \dot{u}|^{q-2}|\nabla \ddot{u}|
\leq C_q |\nabla \dot{u}|^{q-2}|\nabla \ddot{u}|,
\end{align*} and we have used that $\int_0^t \int_{\Omega} b \delta  \,|\nabla \dot{u}|^{q-1} |\nabla \dot{\tilde{p} }|^2 \, dx \, ds \geq 0$ and $\int_0^t \int_{\Omega} (q-1)b\delta  |\nabla \dot{u}|^{q-3}(\nabla \dot{\tilde{p}}\cdot \nabla \dot{u})(\nabla \dot{u} \cdot \nabla \dot{\tilde{p}}) \geq 0$. By taking essential supremum with respect to $t \in [0,T]$ in \eqref{adjoint_estimate} and employing
\begin{align*}
&\int_0^T \int_{\Omega} |\varphi|| \nabla \tilde{p}|  |\nabla \dot{\tilde{p}}| \, dx \, ds \\
\leq& \,C_q\|\nabla \dot{u}\|^{q-2}_{L^\infty(0,T;L^\infty(\Omega))} \|\nabla \ddot{u}\|_{L^2(0,T;L^\infty(\Omega))}\Bigl(\frac{1}{4 \varepsilon}\|\nabla \tilde{p}\|^2_{L^\infty(0,T;L^2(\Omega))}+\varepsilon \|\nabla \dot{\tilde{p}}\|^2_{L^2(0,T;L^2(\Omega))}\Bigr), 
\end{align*}
we get the energy estimate for $p$
\begin{align*}
&\Bigl(\frac{1}{4}(1-a_0)-T\frac{\overline{k}}{\underline{\lambda}}(C^\Omega_{H^1,L^4})^2\|\dot{u}\|_{L^{\infty}(0,T;L^{2 }(\Omega))}\Bigr)\|\dot{\tilde{p}}\|^2_{L^{\infty}(0,T;L^2(\Omega))} \\
&+\Bigl(\frac{1}{4 \overline{\varrho}}-\frac{1}{4\varepsilon}\overline{b}\hspace{0.08em} \overline{\delta} \, C_q\|\nabla \dot{u}\|^{q-2}_{L^\infty(0,T;L^\infty(\Omega))} \|\nabla \ddot{u}\|_{L^2(0,T;L^\infty(\Omega))}\Bigr)\|\nabla \tilde{p}\|^2_{L^{\infty}(0,T;L^2(\Omega))} \\
&+\Bigl(\frac{1}{2}\underline{b}(1-\overline{\delta})-\frac{\overline{k}}{\underline{\lambda}}(C^\Omega_{H^1,L^4})^2\|\dot{u}\|_{L^{\infty}(0,T;L^{2 }(\Omega))} \\
&\quad \quad-\varepsilon \overline{b} \hspace{0.08em} \overline{\delta} \, C_q\|\nabla \dot{u}\|^{q-2}_{L^\infty(0,T;L^\infty(\Omega))} \|\nabla \ddot{u}\|_{L^2(0,T;L^\infty(\Omega))}\Bigr)\|\nabla \dot{\tilde{p}}\|^2_{L^{2}(0,T;L^2(\Omega))}\\
\leq & \int_0^T \int_{\Omega} j^{\prime}(u) \dot{\tilde{p}} \, dx \, ds \leq 2 \|u-u_d\|_{L^2(0,T;L^2(\Omega))}\|\dot{\tilde{p}}\|_{L^2(0,T;L^2(\Omega))},
\end{align*} 
with $1-a_0>0$, defined as in \eqref{avoid_degeneracy}, bounding away from zero factor $1-2ku$. This estimate holds under additional assumptions on smallness of $T$, $\bar{m}$ and $ \|\nabla \dot{u}\|^{q-2}_{L^\infty(0,T;L^\infty(\Omega))} \|\nabla \ddot{u}\|_{L^2(0,T;L^\infty(\Omega))}$:
\begin{align} \label{smallnessmT} 
& T\frac{\overline{k}}{\underline{\lambda}}(C^\Omega_{H^1,L^4})^2C_{\text{P}}\bar{m} < \frac{1}{4}(1-a_0), \nonumber \\
& \frac{1}{\varepsilon}\overline{b}\hspace{0.08em}\overline{\delta}C_q \, \|\nabla \dot{u}\|^{q-2}_{L^\infty(0,T;L^\infty(\Omega))} \|\nabla \ddot{u}\|_{L^2(0,T;L^\infty(\Omega))}< \frac{1}{\overline{\varrho}}, \\
& \frac{\overline{k}}{\underline{\lambda}}(C^\Omega_{H^1,L^4})^2C_\text{P} \bar{m} <\frac{1}{2}\underline{b}(1-\overline{\delta}), \nonumber
\end{align} 
and some sufficiently small $\varepsilon>0$. Here we have employed Poincar\' e's inequality for functions in $H_0^1(\Omega)$ (see Theorem 2.17, \cite{Leoni} ): \begin{align*}\|\dot{u}\|_{L^{\infty}(0,T;L^2(\Omega))} \leq C_\text{P} \|\nabla \dot{u}\|_{L^{\infty}(0,T;L^2(\Omega))} \leq C_{\text{P}}\bar{m}, \quad \ u \in \cW, \end{align*} with $C_{\text{P}}=C_{\text{P}}(\Omega)>0.$ Moreover, we have
\begin{align*}
|\ddot{\tilde{p}}(t)|_{H^{-1}(\Omega)}=& \, \sup_{\tiny \zeta \in H_0^1(\Omega)\setminus \{0\}} \frac{(\ddot{\tilde{p}}(t),\zeta)_{L^2(\Omega)}}{|\zeta|_{H_0^1(\Omega)}} \\
=& \, \int_{\Omega} \lambda \Bigl\{\frac{2k}{\lambda}u\ddot{\tilde{p}}\zeta-\frac{1}{\varrho}\nabla \tilde{p} \cdot \nabla \zeta 
+ b(1-\delta)\nabla \dot{\tilde{p}} \cdot \nabla \zeta \\
&-b \delta  (|\nabla \dot{u}|^{q-1} \nabla \dot{\tilde{p} } 
 -  (|\nabla \dot{u}|^{q-1})^{\Lcdot} \nabla \tilde{p})  \cdot \nabla \zeta  \\
&+(q-1)b\delta \Bigl( (|\nabla \dot{u}|^{q-3})^{\Lcdot} (\nabla \tilde{p}  \cdot \nabla \dot{u})\nabla \dot{u} -  |\nabla \dot{u}|^{q-3}(\nabla \dot{\tilde{p}}\cdot \nabla \dot{u})\nabla \dot{u}   \\
&+  |\nabla \dot{u}|^{q-3}(\nabla \tilde{p} \cdot \nabla \ddot{u})\nabla \dot{u} +  |\nabla \dot{u}|^{q-3}(\nabla \tilde{p} \cdot \nabla \dot{u})\nabla \ddot{u} \Bigr)\cdot \nabla \zeta   
+ j^{\prime}(u) \zeta \Bigr\} \, dx .
\end{align*} 
This further implies
\begin{align*}
|\ddot{\tilde{p}}(t)|_{H^{-1}(\Omega)}\leq& \, C(|u(t)|_{L^\infty(\Omega)}|\ddot{\tilde{p}}(t)|_{H^{-1}(\Omega)}+|\nabla \tilde{p}(t)|_{L^2(\Omega)}|\nabla \dot{u}(t)|^{q-2}_{L^\infty(\Omega)}|\nabla \ddot{u}(t)|_{L^\infty(\Omega)}\\
&+|\nabla \dot{\tilde{p}}(t)|_{L^2(\Omega)}|\nabla \dot{u}(t)|^{q-1}_{L^\infty(\Omega)}+|u(t)-u_d(t)|_{L^2(\Omega)}),
\end{align*}
with $C=C(k,\lambda, \varrho, b,\delta, q) >0$.
By squaring and integrating over $(0,T)$, we can achieve that $\ddot{\tilde{p}} \in L^2(0,T;H^{-1}(\Omega))$. \\
\indent According to Theorem 3, Section 5.9, \cite{evans}, since $\dot{\tilde{p}} \in L^2(0,T;H_0^1(\Omega))$ and $\ddot{\tilde{p}} \in L^2(0,T;H^{-1}(\Omega))$, it follows that $\dot{\tilde{p}} \in C(0,T;L^2(\Omega))$. The statement then comes from reversing the time transformation.
\end{proof}
\section{Existence of optimal shapes} \label{existence_opt_shapes}
From now on, for simplicity of exposition, we assume that all the coefficients in the state equation are piecewise constants, i.e. 
\begin{align} \label{coeff_3}
\begin{cases}
\ \lambda, k, \varrho, b, \delta \in L^\infty(\Omega),\\
w_i:=w\vert_{\Omega_i} \ \text{is constant}, \ \text{for} \ w \in \{b, \varrho, \lambda, \delta, k\}, i \in\{+,-\}, \\
w_i>0 \ \text{for} \ w \in \{b, \varrho, \lambda\}, \ \delta_i \in (0,1), \ k_i \in \mathbb{R}.
\end{cases}
\end{align} 
Note that now $\underline{\omega}=\min\{|\omega_+|,|\omega_-|\}$, $\overline{\omega}=\max\{|\omega_+|,|\omega_-|\}$, where $\omega \in \{b, \varrho, \lambda, \delta, k\}$.\\
\indent We turn next to the question of existence of minimizers for the shape optimization problem \eqref{shape_opt_problem}, with the coefficients satisfying assumptions \eqref{coeff_3}. The main idea of the proof is to employ the Bolzano-Weierstrass theorem for continuous functions on compact sets (for this approach see, for instance, Section 3, \cite{GL} and  Section 3.2, \cite{KK}). \\
\indent  We recall the following compactness result (cf. Theorem 2, \cite{GL}):
\begin{theorem} \label{minimizers}
Let $\Omega^n_+$ be a sequence in $\mathcal{O}_{\text{ad}}$. Then there exists $\Omega^\star_+ \in \mathcal{O}_{\text{ad}}$ and a subsequence $\Omega^{n_k}_+$ which converges to $\Omega^\star_+$ in the sense of Hausdorff, and in the sense of characteristic functions. Additionally, $\overline{\Omega}^+_{n_k}$ and $\partial \Omega^+_{n_k}$ converge in the sense of Hausdorff towards $\overline{\Omega}^\star_+$ and $\partial \Omega^\star_+$, respectively.
\end{theorem}
\indent Here the set of characteristic functions is defined as $$Char(\Omega)=\{\one_{\Omega_+} :\Omega_+\subset \Omega \ \text{is measurable} \ \wedge  \one_{\Omega_+}(1-\one_{\Omega_+})=0 \ \text{a.e. in} \ \Omega \}$$
and convergence on this set first of all means pointwise almost everywhere convergence of functions, but due to Lebesgue's Dominated Convergence Theorem also convergence in $L^p(\Omega)$ for any $p\in[1,\infty)$.\\
\indent We first establish continuity of $u(\Omega_+)$ with respect to the shape of the lens $\Omega_+$.
\begin{proposition} \label{charact_cont}  Let $q \geq 1$, $q >d-1$ and the assumptions \eqref{coeff_3} on the coefficients in \eqref{ModWest} hold. 
Then the mapping $\one_{\Omega_+} \mapsto u(\Omega_+)$ is continuous from the set of characteristic functions $Char(\Omega)$ to $W^{1,\infty}(0,T;L^2(\Omega)) \cap W^{1,q+1}(0,T;W_0^{1,q+1}(\Omega))$.\\
\indent If additionally there exists $\varepsilon >0$ such that the solution $u^\sharp$ of \eqref{ModWest} with $\Omega_+=\Omega_+^\sharp$ satisfies 
\begin{align} \label{reg_condition_u}
\|u^\sharp\|_{W^{1,q+1}(0,T;W^{1,q+1+\varepsilon}(\Omega))} \leq C,
\end{align}
where $C$ depends only on $\Omega$ and the final time $T$, then the mapping $\one_{\Omega_+} \mapsto u(\Omega_+)$ is even H\"older continuous with exponent $\frac{1}{q}$ at $\one_{\Omega_+}^\sharp$ from the set of characteristic functions $Char(\Omega)$ in $L^{r_2}(\Omega)$, $r_2=\frac{(1+q)(1+\varepsilon)}{q\varepsilon}$ to $W^{1,\infty}(0,T;L^2(\Omega)) \cap W^{1,q+1}(0,T;W_0^{1,q+1}(\Omega))$.
\end{proposition}
\begin{proof}
\noindent 
Let $\Omega^n_+$ be a sequence converging to $\Omega^\sharp_+$ in the sense of characteristic functions.
By subtracting the weak forms for $u^n$ and $u^\sharp$, corresponding to the domains with the lens regions $\Omega^n_+$ and $\Omega^\sharp_+$ respectively, we get
\begin{align*}
&\displaystyle \sum_{i \in \{+,- \}}\int_0^T \int_{\Omega} \{\frac{1}{\lambda_i}(1-2k_iu^n)(\ddot{u}^n-\ddot{u}^\sharp) \phi_i-\frac{2k_i}{\lambda_i}(u^n-u^\sharp)\ddot{u}^\sharp\phi_i+\frac{1}{\varrho_i}\nabla (u^n-u^\sharp) \cdot \nabla \phi_i\\
& + b_i(1-\delta_i)\nabla (\dot{u}^n-\dot{u}^\sharp) \cdot \nabla \phi_i 
 +b_i \delta_i (|\nabla \dot{u}^n|^{q-1} \nabla \dot{u}^n-|\nabla \dot{u}^\sharp|^{q-1} \nabla \dot{u}^\sharp) \cdot \nabla \phi\\
 &-\frac{2k_i}{\lambda_i}((\dot{u}^n)^2-(\dot{u}^\sharp_i)^2)\phi\}\, \one_{\Omega^n_i} \, dx \, ds\\
&=\displaystyle \sum_{i \in \{+,- \}}\int_0^T \int_{\Omega} \{\frac{1}{\lambda_i}(1-2k_iu^\sharp)\ddot{u}^\sharp \phi+\frac{1}{\varrho_i}\nabla u^\sharp \cdot \nabla \phi+ b_i(1-\delta_i)\nabla \dot{u}^\sharp \cdot \nabla \phi\\
&   +b_i \delta_i |\nabla \dot{u}^\sharp|^{q-1} \nabla \dot{u}^\sharp \cdot \nabla \phi
 -\frac{2k_i}{\lambda_i}(\dot{u}^\sharp)^2\phi\}\, (\one_{\Omega^n_i}-\one_{\Omega^\sharp_i}) \, dx \, ds,
\end{align*}
with $\Omega^n_-=\Omega \setminus \overline{\Omega}^n_+$, $\Omega^\sharp_-=\Omega \setminus \overline{\Omega}^\sharp_+$. Testing with $\phi=\dot{u}^n-\dot{u}^\sharp$ and employing inequality \eqref{aa_ineq5} for the difference of the $q$-Laplace terms then yields
\begin{align*}
&\frac{1}{2\underline{\lambda}} (1-2\bar{k}\|u^n\|_{L^{\infty}(0,T;L^{\infty}(\Omega))})\|\dot{u}^n-\dot{u}^\sharp\|^2_{L^\infty(0,T;L^2(\Omega))}+\frac{1}{2\underline{\varrho}}\|\nabla(u_n-u_m)\|^2_{L^2(0,T;L^2(\Omega))}\\
&+\underline{b}(1-\overline{\delta})\|\nabla(\dot{u}_n-\dot{u}_m)\|^2_{L^2(0,T;L^2(\Omega))}+2^{1-q}\underline{b}\hspace{0.08em}\underline{\delta}\|\nabla(\dot{u}^n-\dot{u}^\sharp)\|^{q+1}_{L^{q+1}(0,T;L^{q+1}(\Omega))}\\
\leq& \, \frac{2\overline{k}}{\underline{\lambda}}(\|\dot{u}^\sharp\|_{L^2(0,T;L^{\infty}(\Omega))}\|\dot{u}^n-\dot{u}^\sharp\|^2_{L^\infty(0,T;L^2(\Omega))}\\
&+(C_{H_0^{1},L^4}^\Omega)^2\|u^n-u^\sharp\|_{L^\infty(0,T;H_0^1(\Omega))}\|\ddot{u}^\sharp\|_{L^2(0,T;L^2(\Omega))}\|\dot{u}^n-\dot{u}^\sharp\|_{L^2(0,T;H_0^1(\Omega))}\\
&+\displaystyle \sum_{i \in \{+,- \}}\Bigl| \int_0^T \int_{\Omega} \Bigl\{\frac{1}{\lambda_i}(1-2k_iu^\sharp)\ddot{u}^\sharp (\dot{u}^n-\dot{u}^\sharp) +\frac{1}{\varrho_i}\nabla u^\sharp \cdot \nabla (\dot{u}^n-\dot{u}^\sharp)\\
& + b_i(1-\delta_i)\nabla \dot{u}^\sharp \cdot \nabla (\dot{u}^n-\dot{u}^\sharp) 
 +b_i\delta_i|\nabla \dot{u}^\sharp|^{q-1} \nabla \dot{u}^\sharp \cdot \nabla (\dot{u}^n-\dot{u}^\sharp)\\
 &-\frac{2k_i}{\lambda_i}(\dot{u}^\sharp)^2 (\dot{u}^n-\dot{u}^\sharp)\Bigr\} (\one_{\Omega^n_i}-\one_{\Omega^\sharp_i})\, dx \, ds\Bigr|.
\end{align*}
By employing the Sobolev embedding $W_0^{1,q+1}(\Omega) \hookrightarrow L^\infty(\Omega)$:  
$$\|\dot{u}^\sharp\|_{L^2(0,T;L^{\infty}(\Omega))} \leq C^{\Omega}_{W_0^{1,q+1},L^\infty}\|\dot{u}^\sharp\|_{L^2(0,T;W_0^{1,q+1}(\Omega))}$$
and then utilizing \eqref{ModWest_energyest} to estimate the terms $\|\dot{u}^\sharp\|_{L^2(0,T;W_0^{1,q+1}(\Omega))}$ and $\|\ddot{u}^\sharp\|_{L^2(0,T;L^2(\Omega))}$, we conclude that for sufficiently small initial data the two first terms on the right hand side can be absorbed by the appropriate terms on the left hand side. For the remaining terms, we employ the following estimates:
\begin{align*}
& \int_0^T \int_{\Omega} \Bigl\{\frac{1}{\lambda}(1-2k_iu^\sharp)\ddot{u}^\sharp (\dot{u}^n-\dot{u}^\sharp)(\one_{\Omega^n_i}-\one_{\Omega^\sharp_i})\, dx \, ds\\
\leq& \,(1+2\overline{k}\|u^\sharp\|_{L^{\infty}(0,T;L^{\infty}(\Omega))})\|\ddot{u}^\sharp\|_{L^2(0,T;L^2(\Omega))}\|\dot{u}_n-\dot{u}_m\|_{L^2(0,T;L^{q+1}(\Omega))}|\one_{\Omega^n_i}-\one_{\Omega^\sharp_i}|_{L^{r_1}(\Omega)},
\end{align*}
together with
\begin{align*}
&\int_0^T \int_{\Omega} \frac{2k_i}{\lambda_i}(\dot{u}^\sharp)^2 (\dot{u}^n-\dot{u}^\sharp)(\one_{\Omega^n_i}-\one_{\Omega^\sharp_i})\, dx \, ds \\
\leq& \, \frac{2\overline{k}}{\underline{\lambda}}\|\dot{u}^\sharp\|^2_{L^{\infty}(0,T;L^{\infty}(\Omega))}\|\dot{u}^n-\dot{u}^\sharp\|_{L^2(0,T;L^2(\Omega))}|\one_{\Omega^n_i}-\one_{\Omega^\sharp_i}|_{L^2(\Omega)},
\end{align*}
and the estimate
\begin{align*}
&\int_0^T \int_{\Omega}\frac{1}{\varrho}\nabla u^\sharp \cdot \nabla (\dot{u}^n-\dot{u}^\sharp)(\one_{\Omega^n_i}-\one_{\Omega^\sharp_i})\, dx \, ds\\
\leq& \, \frac{1}{\underline{\varrho}}\|\nabla u^\sharp\|_{L^2(0,T;L^2(\Omega))}\|\nabla(\dot{u}^n-\dot{u}^\sharp)\|_{L^{2}(0,T;L^{q+1}(\Omega))}|\one_{\Omega^n_i}-\one_{\Omega^\sharp_i}|_{L^{r}(\Omega)}\\ 
\leq& \, \frac{1}{\underline{\varrho}}T^{\frac{1}{r_1}}\|\nabla u^\sharp\|_{L^2(0,T;L^2(\Omega))}\|\nabla(\dot{u}^n-\dot{u}^\sharp)\|_{L^{q+1}(0,T;L^{q+1}(\Omega))}|\one_{\Omega^n_i}-\one_{\Omega^\sharp_i}|_{L^{r_1}(\Omega)},
\end{align*}
with $r_1=\frac{2(q+1)}{q-1}$. An analogous estimate to the last one can be derived for the $b_i(1-\delta_i)$-term where $\nabla u^\sharp$ is replaced by $\nabla \dot{u}^\sharp$. 
Thus all these terms on the right hand side tend to zero as $n\to\infty$.
\indent
Finally, by Lebesgue's Dominated Convergence Theorem also the integral 
\[
\int_0^T \int_{\Omega} b_i \delta_i|\nabla \dot{u}^\sharp|^{q-1} \nabla \dot{u}^\sharp \cdot \nabla (\dot{u}^n-\dot{u}^\sharp)(\one_{\Omega^n_i}-\one_{\Omega^\sharp_i})\, dx \, ds\,,
\]
whose integrand due to the factor $(\one_{\Omega^n_i}-\one_{\Omega^\sharp_i})$ tends to zero pointwise a.e. and is bounded by the integrable function 
$\overline{b}\hspace{0.08em}\overline{\delta} \|\nabla \dot{u}^\sharp\|^{q}_{L^{q+1}(0,T;L^{q+1+\varepsilon}(\Omega))}(\|\nabla\dot{u}^n\|_{L^{q+1}(0,T;L^{q+1}(\Omega))}+\|\nabla\dot{u}^\sharp\|_{L^{q+1}(0,T;L^{q+1}(\Omega))})\one_{\Omega}$,
goes to zero as $n\to\infty$, which proves the assertion.
\\
If we additionally assume that the solution of the state problem exhibits a slightly higher regularity in space than $W^{1,q+1}(\Omega)$, namely that $u \in W^{1,q+1}(0,T;W^{1,q+1+\varepsilon}(\Omega))$ for some $\varepsilon>0$, we can 
enhance the estimate on the latter term to
\begin{align*}
& \int_0^T \int_{\Omega} b_i \delta_i|\nabla \dot{u}^\sharp|^{q-1} \nabla \dot{u}^\sharp \cdot \nabla (\dot{u}^n-\dot{u}^\sharp)(\one_{\Omega^n_i}-\one_{\Omega^\sharp_i})\, dx \, ds\\
\leq& \,\overline{b}\hspace{0.08em}\overline{\delta} \|\nabla \dot{u}^\sharp\|^{q}_{L^{q+1}(0,T;L^{q+1+\varepsilon}(\Omega))}\|\nabla(\dot{u}^n-\dot{u}^\sharp)\|_{L^{q+1}(0,T;L^{q+1}(\Omega))}|\one_{\Omega^n_i}-\one_{\Omega^\sharp_i}|_{L^{r_2}(\Omega)},
\end{align*}
with $r_2=\frac{(1+q)(1+\varepsilon)}{q\varepsilon}$. Altogether, we can then conclude that
\begin{align*}
& (1-2\overline{k}\|u^n\|_{L^{\infty}(0,T;L^{\infty}(\Omega))})\|\dot{u}^n-\dot{u}^\sharp\|^2_{L^\infty(0,T;L^2(\Omega))}+\|\nabla(u^n-u^\sharp)\|^2_{L^2(0,T;L^2(\Omega))}\\
&+\|\nabla(\dot{u}^n-\dot{u}^\sharp)\|^2_{L^2(0,T;L^2(\Omega))}+\|\nabla(\dot{u}^n-\dot{u}^\sharp)\|^{q+1}_{L^{q+1}(0,T;L^{q+1}(\Omega))}\\
\leq& \, C((1+2\overline{k}\|u^\sharp\|_{L^{\infty}(0,T;L^{\infty}(\Omega))})\|\ddot{u}^\sharp\|_{L^2(0,T;L^2(\Omega))}\|\dot{u}^n-\dot{u}^\sharp\|_{L^2(0,T;L^{q+1}(\Omega))}|\one_{\Omega^n_+}-\one_{\Omega^\sharp_+}|_{L^{r_1}(\Omega)}\\
&+\|\nabla u^\sharp\|_{L^2(0,T;L^2(\Omega))}\|\nabla(\dot{u}^n-\dot{u}^\sharp)\|_{L^{q+1}(0,T;L^{q+1}(\Omega))}|\one_{\Omega^n_+}-\one_{\Omega^\sharp_+}|_{L^{r_1}(\Omega)}\\
&+\|\nabla \dot{u}^\sharp\|_{L^2(0,T;L^2(\Omega))}\|\nabla(\dot{u}^n-\dot{u}^\sharp)\|_{L^{q+1}(0,T;L^{q+1}(\Omega))}|\one_{\Omega^n_+}-\one_{\Omega^\sharp_+}|_{L^{r_1}(\Omega)}\\
&+\|\nabla \dot{u}^\sharp\|^{q}_{L^{q+1}(0,T;L^{q+1+\varepsilon}(\Omega))}\|\nabla(\dot{u}^n-\dot{u}^\sharp)\|_{L^{q+1}(0,T;L^{q+1}(\Omega))}|\one_{\Omega^n_+}-\one_{\Omega^\sharp_+}|_{L^{r_2}(\Omega)}\\
&+\|u^n\|^2_{L^{\infty}(0,T;L^{\infty}(\Omega))}\|\dot{u}_n-\dot{u}_m\|_{L^2(0,T;L^2(\Omega))}|\one_{\Omega^n_+}-\one_{\Omega^\sharp_+}|_{L^2(\Omega)}),
\end{align*}
for sufficiently large $C>0$ independent of $n$, which implies H\"older continuity of the mapping $\one_{\Omega_+} \mapsto u(\Omega_+)$. 
\end{proof}
Now let $\Omega^n_+$  be a minimizing sequence for \eqref{shape_opt_problem}. Due to Theorem \ref{minimizers}, there exists a subsequence, which for brevity we still denote $\Omega^n_+$, that converges to some $\Omega^\star_+ \in \mathcal{O}$. By extracting another sequence, we may as well assume that $\Omega^n_+ \rightarrow \Omega^\star_+$ in the sense of characteristic functions. \\
\indent Let us denote by $u^n$ the weak solution corresponding to the domain where the lens region is given by $\Omega^n_+$. We know then that $u^n$ satisfies the estimate \eqref{ModWest_energyest}, where $u$ is interchanged with $u^n$. This means that we may extract a subsequence, again denoted $u^n$, such that 
\begin{align*}
& u^n \rightharpoonup u^{\star} \ \text{in} \ H^{2}(0,T;L^2(\Omega)), \\
& \nabla \dot{u}^n \rightharpoonup \nabla \dot{u}^{\star} \ \text{ in} \  L^{q+1}(0,T;L^{q+1}(\Omega)).
\end{align*}
Due to the embedding $H^2(0,T) \hookrightarrow C^1(0,T)$, this further implies that $u(t) \rightharpoonup u^\star(t)$ and  $\dot{u}(t) \rightharpoonup \dot{u}^\star(t)$ in $L^2(\Omega)$ for all $t \in [0,T]$. Therefore, $(u^\star,\dot{u}^\star)\vert_{t=0}=(u_0,u_1)$.
\indent It remains to show that $u^\star$ solves the state problem on the domain whose lens region is given by $\Omega^\star_+$. 
\begin{proposition}
Let the assumptions of Proposition  \ref{charact_cont} be satisfied. Let $\Omega^n_+$ be a minimizing sequence for the shape optimization problem \eqref{sho_problem} and let $\Omega^\star_+$ be an accumulation point of this sequence in accordance with Theorem \ref{minimizers}. Then the sequence $u^n$ corresponding to the domain with the lens region $\Omega^n_+$ converges strongly to $u^\star$ in $W^{1,\infty}(0,T;L^2(\Omega)) \cap W^{1,q+1}(0,T;W_0^{1,q+1}(\Omega))$, where $u^\star$ is the solution of \eqref{ModWest} in the domain whose lens region is given by $\Omega^\star_+$.
\end{proposition}
\begin{proof}
In order to see that $u^\star$ is the weak solution of the state problem in the domain where the lens region is given by $\Omega^\star_+$, note that due to the fact that $u^n$ solves \eqref{ModWest} with lens region $\Omega^n_+$ and using integration by parts in the first term on the right hand side
\begin{align} \label{prop:ex_min}
&\displaystyle \sum_{i \in \{+,-\}}\int_0^T \int_{\Omega} \{\frac{1}{\lambda_i}(1-2k_iu^\star)\ddot{u}^\star \phi+\frac{1}{\varrho_i}\nabla u^\star \cdot \nabla \phi + b_i(1-\delta_i)\nabla \dot{u}^\star \cdot \nabla \phi \nn \\
& +b_i\delta_i|\nabla \dot{u}^\star|^{q-1} \nabla \dot{u}^\star\cdot \nabla \phi-\frac{2k_i}{\lambda_i}(\dot{u}^\star)^2 \phi\} \one_{\Omega^\star_i}\, dx \, ds \nn \\
=
&\,\displaystyle \sum_{i \in \{+,-\}} \int_0^T \int_{\Omega} \{-\frac{1}{\lambda_i}(1-2k_iu^\star)(\dot{u}^\star-\dot{u}^n) \dot{\phi}-\frac{2k_i}{\lambda_i}(u^\star-u^n)\ddot{u}^n \phi+\frac{1}{\varrho_i}\nabla (u^\star-u^n) \cdot \nabla \phi \nn \\
& + b_i(1-\delta_i)\nabla (\dot{u}^\star-u^n) \cdot \nabla \phi 
 +b_i \delta_i(|\nabla \dot{u}^\star|^{q-1} \nabla \dot{u}^\star-|\nabla \dot{u}^n|^{q-1} \nabla \dot{u}^n)\cdot \nabla \phi\\
&-\frac{2k_i}{\lambda_i}(\dot{u}^\star-\dot{u}^n)(\dot{u}^\star+\dot{u}^n) \phi\}\, \one_{\Omega^\star_i} \, dx \, ds \nn \\
&+\displaystyle \sum_{i \in \{+,-\}}\int_0^T \int_{\Omega} \{\frac{1}{\lambda_i}(1-2k_iu^n)\ddot{u}^n \phi+\frac{1}{\varrho_i}\nabla u^n \cdot \nabla \phi + b_i(1-\delta_i)\nabla \dot{u}^n \cdot \nabla \phi \nn \\
& +b_i\delta_i|\nabla \dot{u}_n|^{q-1} \nabla \dot{u}_n \cdot \nabla \phi-\frac{2k_i}{\lambda_i}(\dot{u}^n)^2 \phi\} (\one_{\Omega^\star_i}-\one_{\Omega^n_i})\, dx \, ds, \nn
\end{align}
for all $\phi \in C^\infty(0,T; C_0^\infty(\Omega))$, $\phi(T)=0$. The difference of the $q$-Laplace terms on the right hand side can be estimated with the help of inequality \eqref{aa_ineq3} (with $\eta=0$) as follows:
\begin{align*}
&\int_0^T \int_{\Omega} b_i \delta_i(|\nabla \dot{u}^\star|^{q-1} \nabla \dot{u}^\star-|\nabla \dot{u}^n|^{q-1} \nabla \dot{u}^n)\cdot \nabla \phi\one_{\Omega^\star_i} \, dx \, ds \\
\leq& \, C_q\overline{b}\hspace{0.08em}\overline{\delta} \int_0^T \int_{\Omega} |\nabla (\dot{u}^\star-\dot{u}^n)|(|\nabla \dot{u}^\star|^{q-1}+|\nabla \dot{u}^n|^{q-1})|\nabla \phi||\one_{\Omega^\star_i}|\, dx \, ds \\
\leq&\,C_q\overline{b}\hspace{0.08em}\overline{\delta}\|\nabla (\dot{u}^\star-\dot{u}^n)\|_{L^{q+1}(0,T;L^{q+1}(\Omega))}
\Bigl(\|\nabla \dot{u}^\star\|^{q-1}_{L^{q+1}(0,T;L^{q+1}(\Omega))}\\
&+\|\nabla \dot{u}^n\|^{q-1}_{L^{q+1}(0,T;L^{q+1}(\Omega))}\Bigr)\|\nabla \phi\|_{L^{q+1}(0,T;L^{q+1}(\Omega))}.
\end{align*} The remaining terms can be estimated analogously to the estimates in the proof of Proposition \ref{charact_cont}, from which it then follows that the right hand side in \eqref{prop:ex_min} tends to zero as $n \rightarrow \infty$.
\end{proof}
\begin{theorem} The shape optimization problem \eqref{shape_opt_problem} has a solution.
\end{theorem}
\begin{proof}
Let us define the reduced cost functional $\hat{J}:\mathcal{O}_{\text{ad}} \rightarrow \mathbb{R}$:
 \begin{align*}
 \hat{J}(\Omega_+)=J(u(\Omega_+), \Omega_+).
\end{align*} 
Let $\Omega^n_+ \rightarrow \Omega^\star_+$ as $n \rightarrow \infty$. It can be shown (cf. Lemma 3.3, \cite{KK}) that 
\begin{align*}
|J(u^\star,\Omega^\star_{+})-J(u^n,\Omega^n_{+})| \leq \|u^\star-u^n\|_{L^2(0,T;L^2(\Omega))}\|u^\star+u^n-2u_d\|_{L^2(0,T;L^2(\Omega))}.
\end{align*}
Since the term $\|u^\star+u^n-2u_d\|_{L^2(0,T;L^2(\Omega))}$ is uniformly bounded due to \eqref{ModWest_energyest}, by employing Proposition \ref{charact_cont} we achieve that the right hand side tends to zero as $n \rightarrow \infty$. Therefore the cost functional $\hat{J}$ is continuous on $\mathcal{O}_{\text{ad}}$.  According to the Bolzano-Weierstrass theorem, since $\mathcal{O}_{\text{ad}}$ is compact, $J$ attains a global minimum on $\mathcal{O}_{ad}$.
\end{proof}
\section{State equation on the domain with perturbed lens} \label{state_perturbed} 
\subsection{Method of mappings.} The approach to computing the shape derivative that we will take follows the general framework given in \cite{IKP08} and its extension to a time dependent setting given in \cite{KP}. One of its main ingredients is the mapping method, originally introduced by Murat and Simon in \cite{MS}, which we briefly recall. \\
\indent We introduce a fixed vector field $h \in C^{1,1}(\bar{\Omega},\mathbb{R}^d)$ with $h\vert_{\partial \Omega}=0$, and a family of transformations $F_\tau: \Omega \mapsto \mathbb{R}^d$
\begin{align*}
F_\tau=id+\tau h, \ \text{for} \ \tau \in \mathbb{R}.
\end{align*}
There exists $\tau_0>0$, such that for $|\tau|<\tau_0$, $F_\tau$ is a $C^{1,1}$-diffeomorphism (cf. \cite{DZ}). If the perturbed lens is given by $\Omega_{+,\tau}=F_\tau(\Omega_+)$ and $\Gamma_\tau=F_\tau(\Gamma)$, then it follows that $\Gamma_\tau$ is strongly Lipschitz continuous (cf. Theorem 4.1, \cite{Hofmann}). \\
\indent The Eulerian derivative of $J$ at $\Omega_+$ in the direction of the vector field $h$ is defined as
$$dJ(u,\Omega_+)h=\displaystyle \lim_{\tau \rightarrow 0} \frac{1}{\tau}(J(u_\tau,\Omega_{+,\tau})-J(u,\Omega_+)),$$
\noindent where $u_\tau$ satisfies the state equation on the perturbed domain $\Omega_\tau$. The functional $J$ is said to be shape differentiable at $\Omega_+$ if $dJ(u,\Omega_+)h$ exists for all $h \in C^{1,1}(\bar{\Omega},\mathbb{R}^d)$ and defines a continuous linear functional on $C^{1,1}(\bar{\Omega},\mathbb{R}^d)$. \\
We introduce the following notation:
\begin{equation}\label{defIAw}
I_\tau=\text{det}(DF_\tau)\,, \quad A_\tau=(DF_\tau)^{-T}\, .
\end{equation}
where $DF_\tau$ is the Jacobian of the transformation $F_\tau$.
\begin{lemma} \label{IKP08_eq:2dot8}\cite{IKP08} For sufficiently small $\tau_0>0$, mapping $F_\tau$ has the following properties:
\begin{equation*}
\begin{aligned}
&\tau \mapsto F_\tau\in C^1(-\tau_0,\tau_0;C^1(\overline{\Omega},\R^d))
&&\tau\mapsto A_\tau\in C(-\tau_0,\tau_0;C(\overline{\Omega},\R^{d\times d}))\\
&\tau\mapsto F_\tau\in C(-\tau_0,\tau_0;C^{1,1}(\overline{\Omega},\R^d))	
&&\tau\mapsto I_\tau\in C^1(-\tau_0,\tau_0;C(\overline{\Omega}))\\
&\tau\mapsto F_\tau^{-1}\in C(-\tau_0,\tau_0;C^1(\overline{\Omega},\R^d)) 		
&&F_0=\text{id} \\
&\frac{d}{d\tau} F_\tau\vert_{\tau=0}=h		&&\frac{d}{d\tau} F_\tau^{-1}\vert_{\tau=0}=-h\\
&\frac{d}{d\tau} DF_\tau\vert_{\tau=0}=Dh	&&\frac{d}{d\tau} DF_\tau^{-1}\vert_{\tau=0}=\frac{d}{d\tau} (A_\tau)^T\vert_{\tau=0}=-Dh\\
&\frac{d}{d\tau} I_\tau\vert_{\tau=0}=\div h.&& \frac{d}{d\tau} A_\tau\vert_{\tau=0} =-(Dh)^T.
\end{aligned}
\end{equation*}
\end{lemma}
\noindent As a consequence of Lemma \ref{IKP08_eq:2dot8}, there exist $\alpha_0, \alpha_1 >0$ such that 
\begin{align}\label{uniform_bound_I}
&0<\alpha_0 \leq I_\tau(x) \leq \alpha_1,  \quad \quad  \text{for} \  x \in \overline{\Omega}, \ \tau \in [-\tau_0,\tau_0].
\end{align}
Furthermore, there exist $\beta_1, \beta_2>0$ such that
\begin{align} \label{uniform_bound_A}
&  |A_\tau|_{L^\infty(\Omega)} \leq \beta_1,\quad |A^{-1}_\tau|_{L^\infty(\Omega)} \leq \beta_2, \quad  \text{for}  \ \tau \in [-\tau_0,\tau_0].
\end{align}
\noindent We will often employ the following lemma which gives us the formula for the transformation of domain integrals:
\begin{lemma} \label{transf_lemma} \cite{SZ} Let $\varphi_\tau \in L^1(\Omega_\tau)$, then $\varphi_\tau \circ F_\tau \in L^1(\Omega)$ and
$$\int_{\Omega_\tau} \varphi_\tau \, dx_\tau = \int_{\Omega} (\varphi_\tau \circ F_\tau) \, I_\tau \, dx.$$
\end{lemma}
\vspace{2mm}
\indent Let us introduce, for some Lipschitz domain $\Omega_+ \in \mathcal{O}_{\text{ad}}$, the operator $E(\cdot,\Omega_+): \cW \rightarrow \tilde{X}^\star$ by
\begin{equation*}
\begin{aligned}
\langle E(u,\Omega_+), \phi \rangle_{\tilde{X}^{\star},\tilde{X}} =& \,\displaystyle \sum_{i \in \{+,-\}} \int_0^T \int_{\Omega_i} \{\frac{1}{\lambda_i}(1-2k_iu)\ddot{u} \phi+\frac{1}{\varrho_i}\nabla u \cdot \nabla \phi + b_i(1-\delta_i)\nabla \dot{u} \cdot \nabla \phi \\
&+b_i\delta_i|\nabla \dot{u}|^{q-1} \nabla \dot{u} \cdot\nabla \phi -\frac{2k_i}{\lambda_i}(\dot{u})^2 \phi\} \, dx \, ds.
\end{aligned}
\end{equation*}
\noindent Fix $\tau \in (-\tau_0,\tau_0)$. Proposition \ref{aa_corollary} guarantees that $E(u, \Omega_{+,\tau})=0$ has a unique solution, which we denote $u_\tau: \Omega \rightarrow \mathbb{R}$.
\noindent We transport $u_\tau$ back to the domain with the fixed lens $\Omega_+$ by defining $u^\tau: \Omega \rightarrow \mathbb{R}$ as 
\begin{align*}
u^\tau=u_\tau \circ F_\tau.
\end{align*}
Differentiability of $h$ implies that $u^\tau \in X$. 
We further have
\begin{equation}\label{Etil} 
\begin{aligned} 
&\langle E(u_\tau,\Omega_{+,\tau}), \phi_\tau \rangle_{\tilde{X}_\tau^{\star},\tilde{X}_\tau} \\
=& \, \displaystyle \sum_{i \in \{+,-\}} \int_0^T \int_{\Omega_{i,\tau}}  \Bigl\{\frac{1}{\lambda_i}(1-2k_i u_\tau)\ddot{u}_\tau \phi_\tau+\frac{1}{\varrho_i}\nabla u_\tau \cdot \nabla \phi_\tau\\
& + b_i(1-\delta_i)\nabla \dot{u}_\tau \cdot \nabla \phi_\tau 
+b_i\delta_i|\nabla \dot{u}_\tau|^{q-1} \nabla \dot{u}_\tau \cdot \nabla \phi_\tau 
-\frac{2k_i}{\lambda_i}(\dot{u}_\tau)^2 \phi_\tau \Bigr\} \, dx_\tau \, ds  \\
=& \displaystyle \sum_{i \in \{+,-\}} \, \int_0^T \int_{\Omega_i} \Bigl\{\frac{1}{\lambda_i}(1-2k_i u^\tau)\ddot{u}^\tau \phi^\tau+\frac{1}{\varrho_i}A_\tau\nabla u^\tau \cdot A_\tau\nabla \phi^\tau
+ b_i(1-\delta_i)A_\tau\nabla \dot{u}^\tau \cdot A_\tau\nabla \phi^\tau \\
&+b_i \delta_i|A_\tau\nabla \dot{u}^\tau|^{q-1} A_\tau\nabla \dot{u}^\tau \cdot A_\tau\nabla \phi^\tau -\frac{2k^\tau}{\lambda_i}(\dot{u}^\tau)^2 \phi^\tau \Bigr\} \, I_\tau \, dx \, ds  \\
\equiv & \ \langle \tilde{E}(u^\tau,\tau), \phi^\tau \rangle_{\tilde{X}^{\star},\tilde{X}},
\end{aligned}
\end{equation}
for any $\phi_\tau \in \tilde{X}_\tau=\tilde{X}$, with $\Omega_{-,\tau}=\Omega \setminus \overline{\Omega}_{+,\tau}$, where we have used Lemma \ref{transf_lemma} and the fact that $\nabla u_\tau=A_\tau \nabla u^\tau\circ F_\tau^{-1}$. Therefore, for sufficiently small $|\tau|$, $u^\tau$ uniquely satisfies an equation on the domain with the fixed lens:
\begin{align} \label{equation}
\tilde{E}(u^\tau,\tau)=0.
\end{align}
\noindent Since $F_0=id$, we have that $u^0=u$ and $\tilde{E}(u,0)=E(u,\Omega_+  )$. 
\subsection{Continuity of the state with respect to domain perturbations}\noindent We will now focus our attention on the speed of convergence of $u^\tau$ to $u$ as $\tau \rightarrow 0$ and prove two properties which together correspond to hypothesis (H2) in \cite{IKP08} (see also Proposition 3.1 in \cite{IKP}). \\
\indent We begin with the question of uniform boundedness of $u^\tau$ with respect to $\tau \in (-\tau_0,\tau_0)$. Since
\begin{align*}
\|u^\tau\|_{L^{\infty}(0,T;L^{\infty}(\Omega))}=&\, \|u_\tau \circ F_\tau\|_{L^{\infty}(0,T;L^{\infty}(\Omega))} \\
\leq& \, \|u_\tau\|_{L^{\infty}(0,T;L^{\infty}(\Omega))},
\end{align*}
we also have that
\begin{align} \label{avoid_degeneracy_1}
1-a_0 <\|1-2ku^\tau\|_{L^{\infty}(0,T;L^{\infty}(\Omega))}<1+a_0,
\end{align}
where $a_0$ is defined as in \eqref{avoid_degeneracy}. 
\begin{proposition} \label{uniform_bound} Let $q \geq 1$, $q>d-1$ and assumptions \eqref{coeff_3} hold.
Solutions $u^\tau$ of \eqref{equation} are uniformly bounded in $W^{1,\infty}(0,T;L^2(\Omega)) \cap W^{1,q+1}(0,T;W^{1,q+1}_0(\Omega))$ for $\tau \in (-\tau_0,\tau_0)$.
\end{proposition}
\begin{proof}
It can be shown (cf. Lemma 3.3, \cite{IKP}) that
\begin{align*} 
\|\nabla \dot{u}^\tau\|_{L^{q+1}(0,T;L^{q+1}(\Omega))} \leq& \, \frac{1+\tau_0|Dh|_{L^\infty(\Omega)}}{\alpha_0^{1/(q+1)}}\|\nabla \dot{u}_\tau\|_{L^{q+1}(0,T;L^{q+1}(\Omega))}.
\end{align*} 
with $\alpha_0$ as in \eqref{uniform_bound_I}
\noindent This implies that, for $u_\tau \in \cW$,  we can estimate
\begin{align} \label{cont_est}
\|\nabla \dot{u}^\tau\|_{L^{q+1}(0,T;L^{q+1}(\Omega))} \leq \frac{1+\tau_0|Dh|_{L^\infty(\Omega)}}{\alpha_0^{1/(q+1)}}\bar{M}.
\end{align}
Testing \eqref{equation} with $\tilde{E}$ as in \eqref{Etil} with $\phi^\tau=\dot{u}^\tau\one_{[0,t)} \in L^2(0,T;W_0^{1,q+1}(\Omega))$ yields
\begin{align*}
&\frac{1}{2}\Bigl[\int_{\Omega} \frac{1}{\lambda}(1-2k u^\tau)(\dot{u}^\tau)^2I_\tau \, dx\Bigr]_0^t+\frac{1}{2}\Bigl[ \int_\Omega \frac{1}{\varrho}|A_\tau \nabla u^\tau|^2  I_\tau \,dx \Bigr]_0^t \\
&+\int_0^t \int_\Omega b(1-\delta)|A_\tau \nabla \dot{u}^\tau|^2 I_\tau\, dx \, ds+\int_0^t \int_\Omega b \delta|A_\tau \nabla \dot{u}^\tau|^{q+1} I_\tau\, dx \, ds \\
=&\, \int_0^t \int_\Omega \frac{k}{\lambda}(\dot{u}^\tau)^3 I_\tau\, dx \, ds.
\end{align*}
By taking the supremum over $t \in [0,T]$ and by utilizing the uniform boundedness properties \eqref{uniform_bound_I}, \eqref{uniform_bound_A} and estimates \eqref{avoid_degeneracy_1} and \eqref{cont_est}, we find that
\begin{align}\label{lower_est}
&\frac{1}{4}\frac{\alpha_0}{\overline{\lambda}}(1-a_0)\|\dot{u}^\tau\|^2_{L^\infty(0,T;L^2(\Omega))}+\frac{1}{4}\frac{1}{\overline{\lambda}}\frac{1}{\beta_2^2} \alpha_0 \|\nabla u^\tau\|^2_{L^\infty(0,T;L^2(\Omega))} \nn\\
&+\frac{1}{2}\underline{b}(1-\overline{\delta})\frac{1}{\beta_2^2} \alpha_0\|\nabla \dot{u}^\tau\|^2_{L^2(0,T;L^2(\Omega))}+\frac{1}{2}\underline{b}\hspace{0.08em}\underline{\delta}\frac{1}{\beta_2^{q+1}}\alpha_0\|\nabla \dot{u}^\tau\|^{q+1}_{L^{q+1}(0,T;L^{q+1}(\Omega))} \nn \\
\leq& \, \frac{\overline{k}}{\underline{\lambda}}\alpha_1\|\dot{u}^\tau\|_{L^2(0,T;L^\infty(\Omega))}\sqrt{T}\|\dot{u}^\tau\|^2_{L^\infty(0,T;L^2(\Omega))}+\frac{1}{2\underline{\lambda}}(1+a_0)\alpha_1|u_1|^2_{L^2(\Omega)}\nn \\
&+\frac{1}{2\underline{\varrho}}\alpha_1|\nabla u_0|^2_{L^2(\Omega)}\\
\leq& \,\frac{\overline{k}}{\underline{\lambda}}\alpha_1 C^\Omega_{W_0^{1,q+1},L^\infty}T^{\frac{q}{q+1}}\frac{1+\tau_0|Dh|_{L^\infty(\Omega)}}{\alpha_0^{1/(q+1)}}\bar{M}\|\dot{u}^\tau\|^2_{L^\infty(0,T;L^2(\Omega))}+\frac{1}{2\underline{\lambda}}(1+a_0)\alpha_1|u_1|^2_{L^2(\Omega)}\nn\\
&+\frac{1}{2\underline{\varrho}}\alpha_1|\nabla u_0|^2_{L^2(\Omega)}. \nn
\end{align}
From here, for sufficiently small $\bar{M}$ and $T$, we can achieve that the first term on the right hand side gets absorbed by the appropriate term on the left hand side. Thus we have uniform boundedness of $u^\tau$ in $W^{1,\infty}(0,T;L^2(\Omega)) \cap W^{1,q+1}(0,T;W^{1,q+1}_0(\Omega))$, $|\tau| <\tau_0$. 
\end{proof}
\noindent The H\" older continuity of $u$ with respect to domain perturbations is established by our next theorem.
\begin{proposition} \label{prop_h2}
Let $q \geq 1$, $q >d-1$ and let assumptions \eqref{coeff_3} hold. Then 
\begin{equation} \label{prop_h2_limit}
\begin{aligned}
&\displaystyle \lim_{\tau \rightarrow 0}\frac{1}{\tau} \Bigl(\|\dot{u}^\tau-\dot{u}\|^2_{L^{\infty}(0,T;L^2(\Omega))}+\|\nabla (u^\tau-u)\|^2_{L^{\infty}(0,T;L^2(\Omega))} \\
&+\|\nabla (\dot{u}^\tau-\dot{u})\|^2_{L^{2}(0,T;L^2(\Omega))} +\|\nabla (\dot{u}^\tau-\dot{u})\|^{q+1}_{L^{q+1}(0,T;L^{q+1}(\Omega))} \Bigr) = 0.
\end{aligned}
\end{equation}
\end{proposition}
\begin{proof}
Note that the difference $v^\tau=u-u^\tau$ satisfies
\begin{equation}   \label{weakform_difference}
\begin{aligned}
 &\displaystyle \sum_{i \in \{+,-\}}\int_0^T \int_{\Omega_i} \Bigl\{\frac{1}{\lambda_i}(1-2k_iu^\tau_i)\ddot{v}^\tau_i\phi_i I_\tau-\frac{2k_i}{\lambda_i}v^\tau_i\ddot{u}_i\phi_iI_\tau +\frac{1}{\varrho_{i}}\nabla v^\tau_{i}\cdot \nabla \phi_i  \\
&
+b_{i}(1-\delta_{i})\nabla \dot{v}^\tau_{i}\cdot \nabla \phi_i+b_{i}\delta_{i}(|\nabla \dot{u}_{i}|^{q-1}\nabla \dot{u}_i-|\nabla \dot{u}^{\tau}_i|^{q-1}\nabla \dot{u}^{\tau}_i)\cdot \nabla \phi_i  \\
&-\frac{2k_i}{\lambda_i}(\dot{u}_i+\dot{u}^\tau_i)\dot{v}^\tau_i\phi_i\Bigr\} \, dx \, ds \\
=& \, \langle f_+, \phi_+ \rangle_{\tilde{X}^\star_+,\tilde{X}_+}+\langle f_-, \phi_- \rangle_{\tilde{X}^\star_-,\tilde{X}_-},
\end{aligned}
\end{equation}
for all $\phi\in\tilde{X}$, with the two terms on the right hand side given by 
\begin{align*}
&\langle f_i, \phi_i \rangle_{\tilde{X}^\star_i,\tilde{X}_i} \\
 =&\, \int_0^T \int_{\Omega_i} \Bigl\{\frac{1}{\lambda_i}(I_\tau-1)(1-2k_i u_i)\ddot{u}_i\phi_i 
+\frac{1}{\varrho_i}\Bigl((I_\tau-1)A_\tau\nabla u^{\tau}_i \cdot A_\tau \nabla \phi_i\\
&+(A_\tau-I)\nabla u^\tau_i \cdot  A_\tau \nabla \phi_i +\nabla u^{\tau}_i \cdot (A_\tau-I)\nabla \phi_i\Bigr) \\
&+b_i(1-\delta_i)\Bigl((I_\tau-1)A_\tau\nabla \dot{u}^\tau_i \cdot A_\tau \nabla \phi_i+(A_\tau-I)\nabla \dot{u}^\tau_i \cdot A_\tau \nabla \phi_i  
+ \nabla \dot{u}^\tau_i \cdot (A_\tau-I)\nabla \phi_i\Bigr)\\
&+b_i\delta_i\Bigl((|A_\tau\nabla \dot{u}^\tau_i|^{q-1}A_\tau\nabla \dot{u}^\tau_i-|\nabla \dot{u}^\tau_i|^{q-1}\nabla \dot{u}^\tau_i)\cdot \nabla \phi_i
+|A_\tau\nabla \dot{u}^\tau_i|^{q-1}A_\tau\nabla \dot{u}^\tau_i \cdot (A_\tau-I) \nabla \phi_i\\
&+(I_\tau-1) |A_\tau\nabla \dot{u}^\tau_i|^{q-1}A_\tau\nabla \dot{u}^\tau_i \cdot A_\tau\nabla \phi_i \Bigr)
-(I_\tau-1)\frac{2k_i}{\lambda_i}(\dot{u}_i^\tau)^2\phi_i\Bigr\} \, dx \, ds,
\end{align*}
\noindent where $\tilde{X}_i=L^2(0,T;W^{1,q+1}(\Omega_i))$. 
Since $\phi_i=\dot{v}^\tau_i\one_{[0,t)}\in\tilde{X}$, we can use it
as test functions in \eqref{weakform_difference}, which together with the uniform boundedness of $I_\tau$ results in 
\begin{align*}
& \displaystyle \sum_{i \in \{+,-\}} \Bigl\{\Bigl[\frac{1}{2}\alpha_0\int_{\Omega_i}\frac{1}{\lambda_i}(1-2k_iu^\tau_i)(\dot{v}^\tau_i)^2 \, dx+\frac{1}{2}\int_{\Omega_i} \frac{1}{\varrho_i}|\nabla v^\tau_i|^2 \, dx\Bigr]_0^t\\
&  +\int_0^t \int_{\Omega_i} b_i(1-\delta_i)|\nabla \dot{v}^\tau_i|^2 \, dx \, ds +\int_0^t \int_{\Omega_i} b_i\delta_i 2^{1-q} \, |\nabla \dot{v}^\tau_i|^{q+1} \, dx \, ds\Bigr\}\\
\leq& \, \displaystyle \sum_{i \in \{+,-\}} \Bigl\{ \frac{\overline{k}}{\underline{\lambda}}\Bigl[\alpha_1\|\dot{u}^\tau_i\|_{L^2(0,T;L^\infty(\Omega_i))}\|\dot{v}^\tau_i\|^2_{L^\infty(0,T;L^2(\Omega_i))}\\
&+2(C^{\Omega_i}_{H^1,L^4})^2\alpha_1\|\ddot{u}_i\|_{L^{2}(0,T;L^2(\Omega_i))}\|v^\tau_i\|_{L^\infty(0,T;H^1(\Omega_i))}\|\dot{v}_i\|_{L^2(0,T;H^1(\Omega_i))}\\
&+2(\|\dot{u}^\tau_i\|_{L^{2}(0,T;L^{\infty}(\Omega_i))}+\|\dot{u}_i\|_{L^{2}(0,T;L^{\infty}(\Omega_i))})\|\dot{v}^\tau_i\|^2_{L^{\infty}(0,T;L^2(\Omega_i))}\Bigr]+|\langle f_i, \dot{v}^\tau_i \rangle_{\tilde{X}^\star_i,\tilde{X}_i}| \Bigr\}.
\end{align*}
Here $C^{\Omega_i}_{H^1,L^4}$ stands for the norm of the embedding $H^1(\Omega_i) \hookrightarrow L^4(\Omega_i)$ and we have also utilized the fact that $(|\nabla \dot{u}_{i}|^{q-1}\nabla \dot{u}_i-|\nabla \dot{u}^{\tau}_i|^{q-1}\nabla \dot{u}^{\tau}_i)\cdot \nabla (\dot{u}_i-\dot{u}^\tau_i)  \geq 2^{1-q} \, |\nabla (\dot{u}_i-\dot{u}^\tau_i)|^{q+1}$, which follows from \eqref{aa_ineq5}. \\

\noindent We can employ 
\begin{align} \label{uniform_bound1} 
\|\dot{u}^\tau_i\|_{L^{2}(0,T;L^{\infty}(\Omega_i))} \leq  T^{\frac{q-1}{2(q+1)}}C^{\Omega_i}_{W^{1,q+1},L^\infty}\|\dot{u}^\tau_i\|_{L^{q+1}(0,T;W^{1,q+1}(\Omega_i))},
\end{align} and the same inequality with $\dot{u}_i$ instead of $\dot{u}^\tau_i$, as well as
\begin{align*}
\|v^\tau_{i}\|_{L^{\infty}(0,T;H^1(\Omega_{i}))} \leq& \, \sqrt{T} \|\dot{v}^\tau_{i}\|_{L^{2}(0,T;H^1(\Omega_{i}))}, 
\end{align*} (since $v_i^\tau \vert_{t=0}=\dot{v}_i^\tau \vert_{t=0}=0$) to conclude that, because  $\|\dot{u}^\tau_i\|_{L^{q+1}(0,T;W^{1,q+1}(\Omega_i))}$ is uniformly bounded for $\tau \in (-\tau_0,\tau_0)$, for sufficiently small $\bar{m}$ and final time $T$ it holds
\begin{equation} \label{convergence1}
\begin{aligned}
&\displaystyle \sum_{i \in \{+,-\}}(\|\dot{v}^\tau\|^2_{L^{\infty}(0,T;L^2(\Omega))}+\|\nabla v^\tau\|^2_{L^{\infty}(0,T;L^2(\Omega))} 
+\|\nabla \dot{v}^\tau\|^2_{L^{2}(0,T;L^2(\Omega))}\\
& +\|\nabla \dot{v}^\tau\|^{q+1}_{L^{q+1}(0,T;L^{q+1}(\Omega))} )\\
\leq& \, C(|\langle f_+, \dot{v}^\tau_+\rangle_{\tilde{X}^\star_+,\tilde{X}_+}|+|\langle f_-, \dot{v}^\tau_-\rangle_{\tilde{X}^\star_-,\tilde{X}_-}|),
\end{aligned}
\end{equation}
for some sufficiently large $C>0$ which does not depend on $\tau$. By employing the inequality \eqref{aa_ineq3} (with $\eta=0$) we obtain 
\begin{equation} \label{ineq_qL}
\begin{aligned}
&|(|A_\tau\nabla \dot{u}^\tau_i|^{q-1}A_\tau\nabla \dot{u}^\tau_i-|\nabla \dot{u}^\tau_i|^{q-1}\nabla \dot{u}^\tau_i )\nabla \dot{v}_i|  \\
\leq& \, C_q\,|(A_\tau-I)\nabla \dot{u}^\tau_i|\,(|\nabla \dot{u}^\tau_i|^{q-1}+|(A_\tau-I)\nabla \dot{u}^\tau_i|^{q-1})\,|\nabla \dot{v}_i|. 
\end{aligned}
\end{equation}
We can therefore estimate the terms on the right hand side in \eqref{convergence1} as
\begin{align} \label{ineq_qL1}
&|\langle f_i, \dot{v}^\tau_i \rangle_{\tilde{X}^\star_i,\tilde{X}_i }|  \nn \\
 \leq& \, \frac{1}{\overline{\lambda}}|I_\tau-1|_{L^\infty(\Omega_i)}(1+a_0)\|\ddot{u}_i\|_{L^2(0,T;L^2(\Omega_i))}\|\dot{v}^\tau_i\|_{L^2(0,T;L^2(\Omega_i))} \nn \\
&+(|I_\tau-1|_{L^\infty(\Omega_i)}\beta_1^2 +|A_\tau-I|_{L^\infty(\Omega_i)}(1+\beta_1)) 
\Bigl(\frac{1}{\underline{\varrho}}\|\nabla u_i^\tau\|_{L^2(0,T;L^2(\Omega_i))}\nn \\
&+\overline{b}(1-\underline{\delta})\|\nabla \dot{u}_i^\tau\|_{L^2(0,T;L^2(\Omega_i))}\Bigr)\|\nabla \dot{v}^\tau_i\|_{L^2(0,T;L^2(\Omega_i))} \\
&+\overline{b}\hspace{0.08em}\overline{\delta}\Bigl(q|A_\tau-I|_{L^\infty(\Omega_i)}(1+|A_\tau-I|^{q-1}_{L^\infty(\Omega_i)})
+|A_\tau-I|_{L^\infty(\Omega_i)}\beta_1^q\nn\\
&+|I_\tau-1|_{L^\infty(\Omega_i)}\beta_1^{q+1}\Bigr)\|\nabla \dot{u}_i^\tau\|^q_{L^{q+1}(0,T;L^{q+1}(\Omega_i))}\|\nabla \dot{v}^\tau_i\|_{L^{q+1}(0,T;L^{q+1}(\Omega_i))}\nn\\
&+|I_\tau-1|_{L^\infty(\Omega_i)}\frac{2\overline{k}}{\underline{\lambda}}(C^{\Omega_i}_{H^1,L^4})^2\|\dot{u}_i^\tau\|^2_{L^2(0,T;H^1(\Omega_i))}\|\dot{v}^\tau_i\|_{L^{\infty}(0,T;L^{2}(\Omega_i))},\nn
\end{align}
where $a_0$ is defined as in \eqref{avoid_degeneracy}. By inserting this into the estimate \eqref{convergence1} and by employing the uniform boundedness result from Proposition \ref{uniform_bound}, properties of the mapping $F_\tau$ from Lemma \ref{IKP08_eq:2dot8} and Young's inequality, we can conclude that
\begin{equation} \label{convergence2}
\begin{aligned}
&\displaystyle \lim_{\tau \rightarrow 0}\displaystyle \sum_{i \in \{+,-\}}(\|\dot{v}^\tau\|^2_{L^{\infty}(0,T;L^2(\Omega))}+\|\nabla v^\tau\|^2_{L^{\infty}(0,T;L^2(\Omega))} 
+\|\nabla \dot{v}^\tau\|^2_{L^{2}(0,T;L^2(\Omega))}\\
& +\|\nabla \dot{v}^\tau\|^{q+1}_{L^{q+1}(0,T;L^{q+1}(\Omega))} )=0.
\end{aligned}
\end{equation} In oder words, we know that $\displaystyle \lim_{\tau \rightarrow 0} u^\tau = u$ in $X$. To obtain the statement of the Proposition, we divide \eqref{convergence1} by $\tau$, and then it remains to show that $$\displaystyle \lim_{\tau \rightarrow 0} \frac{|\langle f_+, \dot{v}_+\rangle_{\tilde{X}^\star_+,\tilde{X}_+}|+|\langle f_-, \dot{v}_-\rangle_{\tilde{X}^\star_-,\tilde{X}_-}|}{\tau}=0.$$ This now follows from the estimate \eqref{ineq_qL1}, Lemma \ref{IKP08_eq:2dot8}, Proposition \ref{uniform_bound} and \eqref{convergence2}.
\end{proof}
\noindent If we assume higher spatial regularity of $u$, we can even obtain Lipschitz continuity with respect to domain perturbations.
\begin{proposition} \label{prop:un_bound} Let $q \geq 1$, $q>d-1$ and let assumptions \eqref{coeff_3} hold. Assume that the solution $u$ of \eqref{ModWest} satisfies 
\begin{align} \label{prop:un_bound_assumpt}
\|\nabla \dot{u}\|_{L^{2q}(0,T;L^{2q}(\Omega))} \leq \tilde{C},
\end{align} where $\tilde{C}$ depends only on $\Omega$ and the final time $T$. Then $$\frac{1}{\tau}(\|\dot{u}^\tau-\dot{u}\|_{L^{\infty}(0,T;L^2(\Omega_i))}+\|\nabla (u^\tau-u)\|_{L^{\infty}(0,T;L^2(\Omega))}
+\|\nabla (\dot{u}^\tau-\dot{u})\|_{L^{2}(0,T;L^2(\Omega))}) \leq C$$  for all $\tau \in (-\tau_0,\tau_0)$, $\tau \neq 0$, where $C$ does not depend on $\tau$.
\end{proposition}
\begin{proof}
We can rewrite the norm on $\hat{X}=C^1(0,T;L^2(\Omega)) \cap H^1(0,T;H^1_0(\Omega))$ (cf. Proposion \ref{prop:adj}) as
\begin{align*}
\|u\|_{\hat{X}}:=\Bigl(\displaystyle \sum_{i \in \{+,-\}} \|u_i\|^2_{\hat{X}_i}\Bigr)^{1/2},
\end{align*}
where
\begin{align*}
\|u\|_{\hat{X}_i}:=& \,\Bigl(\|\dot{u}_i\|^2_{L^{\infty}(0,T;L^2(\Omega_i))}+\|\nabla u_i\|^2_{L^{\infty}(0,T;L^2(\Omega_i))}+\|\nabla \dot{u}_i\|^2_{L^{2}(0,T;L^2(\Omega_i))}\Bigr)^{1/2}.
\end{align*}
By employing assumption \eqref{prop:un_bound_assumpt}, we can modify estimate \eqref{ineq_qL1} by changing the second to last line as follows:
\begin{align} \label{ineq_qL1_new}
&|\langle f_i, \dot{v}^\tau_i \rangle_{\tilde{X}^\star_i,\tilde{X}_i }|  \nn \\
 \leq& \, \frac{1}{\overline{\lambda}}|I_\tau-1|_{L^\infty(\Omega_i)}(1+a_0)\|\ddot{u}_i\|_{L^2(0,T;L^2(\Omega_i))}\|\dot{v}^\tau_i\|_{L^2(0,T;L^2(\Omega_i))} \nn \\
&+(|I_\tau-1|_{L^\infty(\Omega_i)}\beta_1^2 +|A_\tau-I|_{L^\infty(\Omega_i)}(1+\beta_1)) 
(\frac{1}{\underline{\varrho}}\|\nabla u_i^\tau\|_{L^2(0,T;L^2(\Omega_i))}\nn \\
&+\overline{b}(1-\underline{\delta})\|\nabla \dot{u}_i^\tau\|_{L^2(0,T;L^2(\Omega_i))})\|\nabla \dot{v}^\tau_i\|_{L^2(0,T;L^2(\Omega_i))} \\
&+\overline{b}\hspace{0.08em}\overline{\delta}(q|A_\tau-I|_{L^\infty(\Omega_i)}(1+|A_\tau-I|^{q-1}_{L^\infty(\Omega_i)})
+|A_\tau-I|_{L^\infty(\Omega_i)}\beta_1^q\nn\\
&+|I_\tau-1|_{L^\infty(\Omega_i)}\beta_1^{q+1})\|\nabla \dot{u}_i^\tau\|^q_{L^{2q}(0,T;L^{2q}(\Omega_i))}\|\nabla \dot{v}^\tau_i\|_{L^{2}(0,T;L^{2}(\Omega_i))}\nn\\
&+|I_\tau-1|_{L^\infty(\Omega_i)}\frac{2\overline{k}}{\underline{\lambda}}(C^{\Omega_i}_{H^1,L^4})^2\|\dot{u}_i^\tau\|^2_{L^2(0,T;H^1(\Omega_i))}\|\dot{v}^\tau_i\|_{L^{\infty}(0,T;L^{2}(\Omega_i))}. \nn
\end{align}
This further implies that
\begin{align*}
\|v^\tau\|^2_{\hat{X}} \leq& \, \displaystyle \sum_{i \in \{+,-\}} |\langle f_i, \dot{v}^\tau_i \rangle_{\tilde{X}^\star_i,\tilde{X}_i }| \\
 \leq& \, C(|I_\tau-1|_{L^\infty(\Omega_i)}+|A_\tau-I|_{L^\infty(\Omega_i)}) \|v^\tau_i\|_{\hat{X}},
\end{align*}
where $C>0$ does not depend on $\tau$, from which we can conclude that 
\begin{align*}
\|v^\tau\|_{\hat{X}} \leq
C(|I_\tau-1|_{L^\infty(\Omega_i)}+|A_\tau-I|_{L^\infty(\Omega_i)}).
\end{align*}
The assertion then follows by applying Lemma \ref{IKP08_eq:2dot8}.
\end{proof}
\section{Auxiliary results} \label{preliminary}
\vspace{2mm}
\indent In order to calculate the shape derivative of our cost functional, we will need to employ the two forthcoming propositions. The assertions correspond to hypotheses (H4) and (H3) in \cite{IKP08}. \\
Note that, since
\begin{equation} \label{un_gradient_bound}
\begin{aligned}
\|\nabla \dot{u}^\tau\|_{L^{\infty}(0,T;L^{\infty}(\Omega))}=&\, \|DF_\tau \nabla (u_\tau \circ F_\tau)\|_{L^{\infty}(0,T;L^{\infty}(\Omega))} \\
\leq& \, (1+\tau_0|Dh|_{L^{\infty}(\Omega)})\|\nabla \dot{u}_\tau\|_{L^{\infty}(0,T;L^{\infty}(\Omega))},
\end{aligned}
\end{equation}
if we assume that 
\begin{align*}
(\mathcal{H}_2) \quad \|\nabla \dot{u}\|_{L^{\infty}(0,T;L^{\infty}(\Omega))} \leq C,
\end{align*}
where $C$ depends only on the (fixed) domain $\Omega$ and final time $T$, 
we also know that $\|\nabla u^\tau\|_{L^{\infty}(0,T;L^{\infty}(\Omega))}$ is uniformly bounded for $|\tau|<\tau_0$. Then also condition \eqref{prop:un_bound_assumpt} holds.
\begin{proposition} \label{prop_h4} Assume that the coefficients in the state equation satisfy \eqref{coeff_3} and $q>2$. Let hypotheses \text{\normalfont ($\mathcal{H}_1$)} and \text{\normalfont ($\mathcal{H}_2$)} hold. Then
$$\lim_{t\to0}\frac{1}{\tau} \langle (\tilde{E}(u^\tau,\tau)-\tilde{E}(u,\tau)) - (E(u^\tau,\Omega_+) 
- E(u,\Omega_+)) , p\rangle_{\hat{X}^{\star},\hat{X}}=0
$$
holds for the adjoint state $p$.
\end{proposition}
\begin{proof}
We begin by calculating the difference
\begin{align*}
&\frac{1}{\tau}\langle \tilde{E}(u^\tau,\tau)-\tilde{E}(u,\tau)-(E(u^\tau,\Omega)-E(u,\Omega)), p \rangle_{\hat{X}^{\star},\hat{X}} \\
=& \, \frac{1}{\tau} \displaystyle \sum_{i \in \{+,-\}}\Bigl(I_{i}+II_{i}+III_{i}+IV_{i}\Bigr),
\end{align*}
where we use the following notation
\begin{align*}
I_i=&\, \int_0^T \int_{\Omega_i}\frac{1}{\lambda_i}\Bigr ((1-2k_iu_i^\tau)(\ddot{u}_i^\tau-\ddot{u}_i)-2k_i(u_i^\tau-u_i)\ddot{u}_i \Bigr)(I_\tau-1) p_i \, dx \, ds \\
=&\,
-\int_0^T \int_{\Omega_i}\frac{1}{\lambda_i}(\dot{u}_i^\tau-\dot{u}_i)(I_\tau-1)(-2k_i\dot{u}_i^\tau p_i+(1-2k_iu_i^\tau)\dot{p}_i)\, dx \, ds\\
&-\int_0^T \int_{\Omega_i}\frac{2k_i}{\lambda_i}(u_i^\tau-u_i)\ddot{u}_i (I_\tau-1) p_i \, dx \, ds,
\end{align*}
\begin{align*}
II_i=& \, \int_0^T \int_{\Omega_i}\Bigr\{\Bigr((A_\tau-I)(\frac{1}{\varrho_i}\nabla(u_i^\tau-u_i)+b_i(1-\delta_i)\nabla(\dot{u}_i^\tau-\dot{u}_i)) \\
&+(I_\tau-1)(\frac{1}{\varrho_i}A_\tau\nabla(u_i^\tau-u_i)+b_i(1-\delta_i)A_\tau\nabla(\dot{u}_i^\tau-\dot{u}_i)) \Bigl)\cdot A_\tau\nabla p_i \\
&+\Bigr(\frac{1}{\varrho_i}\nabla(u_i^\tau-u_i)+b_i(1-\delta_i)\nabla(\dot{u}_i^\tau-\dot{u}_i)\Bigl)\cdot (A_\tau-I)\nabla p_i \Bigl\}\, dx \, ds, 
\end{align*}
\begin{align*}
III_i=&\int_0^T \int_{\Omega_i}b_i\delta_i  \Bigl \{ \Bigl(|A_\tau \nabla \dot{u}_i^\tau|^{q-1}A_\tau \nabla \dot{u}_i^\tau- |A_\tau \nabla \dot{u}_i|^{q-1}A_\tau \nabla \dot{u}_i \Bigr)\cdot A_\tau\nabla p_i I_\tau\\
&-\Bigl(|\nabla \dot{u}_i^\tau|^{q-1}\nabla \dot{u}_i^\tau-|\nabla \dot{u}_i|^{q-1} \nabla \dot{u}_i \Bigr)\cdot \nabla p_i\Bigr\}\, dx \, ds,
\end{align*}
\begin{align*}
IV_{i}=-\int_0^T \int_{\Omega_i}\frac{2k_i}{\lambda_i}(\dot{u}_i^\tau-\dot{u}_i)(\dot{u}_i^\tau+\dot{u}_i)(I_\tau-1)p_i\, dx \, ds,
\end{align*}
$i \in \{+,-\}$. Thanks to hypothesis ($\mathcal{H}_1$), we can estimate the first integral as
\begin{align*}
|I_i| \leq&\, \Bigl\{ \frac{1}{\underline{\lambda}}\|\dot{u}_i^\tau-\dot{u}_i\|_{L^{2}(0,T;L^{2}(\Omega))}(2\overline{k}\|\dot{u}_i^\tau\|_{L^2(0,T;L^\infty(\Omega_i))}\|p_i\|_{L^\infty(0,T;L^2(\Omega))} 
+(1+a_0)\|\dot{p}_i\|_{L^2(0,T;L^2(\Omega))})\\
&+\frac{2\overline{k}}{\underline{\lambda}}\|u_i^\tau-u_i\|_{L^{2}(0,T;L^{2}(\Omega))}\|\ddot{u}_i\|_{L^\infty(0,T;L^\infty(\Omega_i))}\|p_i\|_{L^2(0,T;L^2(\Omega))}\Bigr\}|I_\tau-1|_{L^{\infty}(\Omega)},
\end{align*} 
with $a_0$ defined as in \eqref{avoid_degeneracy}. Integrals $II_i$ and $IV_i$ can be estimated in a similar manner and, by employing the uniform boundedness of $A_\tau$, the H\" older continuity result given in Proposition \ref{prop_h2} and properties of the mapping $F_\tau$, 
 we can conclude that $$ \displaystyle \frac{1}{\tau}  \sum_{i \in \{+,-\}} \Bigl(I_{i}+II_{i}+IV_{i}\Bigr) \rightarrow 0, \ \text{as} \ \tau \rightarrow 0.$$ \noindent In order to show convergence of the remaining terms to zero, we first rewrite $III_i$ as
\begin{align} \label{iii}
III_i=& \,\int_0^T \int_{\Omega_i} b_i\delta_i  \Bigl \{ \Bigl(|A_\tau \nabla \dot{u}_i^\tau|^{q-1}A_\tau \nabla \dot{u}_i^\tau- |A_\tau \nabla \dot{u}_i|^{q-1}A_\tau \nabla \dot{u}_i \nn \\
&-(|\nabla \dot{u}_i^\tau|^{q-1}\nabla \dot{u}_i^\tau-|\nabla \dot{u}_i|^{q-1} \nabla \dot{u}_i)\Bigr)\cdot A_\tau\nabla p_i I_\tau \\
&+ (|\nabla \dot{u}_i^\tau|^{q-1}\nabla \dot{u}_i^\tau-|\nabla \dot{u}_i|^{q-1} \nabla \dot{u}_i )\cdot (I_\tau A_\tau-I) \nabla p_i \Bigr\} \, dx \, ds. \nn
\end{align}
By employing inequality \eqref{aa_ineq3} (with $\eta=0$) and hypothesis $(\mathcal{H}_2)$ we can estimate the last line as follows 
\begin{align*}
&\Bigl|\int_0^T \int_{\Omega_i}b_i \delta_i(|\nabla \dot{u}_i^\tau|^{q-1}\nabla \dot{u}_i^\tau-|\nabla \dot{u}_i|^{q-1} \nabla \dot{u}_i)\cdot (I_\tau A_\tau-I) \nabla p_i \, dx \, ds\Bigr|\\
\leq& \, C_q\overline{b}\hspace{0.08em} \overline{\delta}\int_0^T \int_{\Omega_i}|\nabla (\dot{u}_i^\tau-\dot{u}_i)|(|\nabla \dot{u}_i|^{q-1}+|\nabla \dot{u}_i^\tau|^{q-1})|(I_\tau A_\tau-I) \nabla p_i|\, dx \, ds, \\
\leq& \, C_q\overline{b}\hspace{0.08em} \overline{\delta} \|\nabla (\dot{u}_i^\tau-\dot{u}_i)\|_{L^{2}(0,T;L^{2}(\Omega))}(\|\nabla \dot{u}_i\|^{q-1}_{L^{\infty}(0,T;L^{\infty}(\Omega))}\\
&+\|\nabla \dot{u}_i^\tau\|^{q-1}_{L^{\infty}(0,T;L^{\infty}(\Omega))})|I_\tau A_\tau-I|_{L^{\infty}(\Omega_i)} \|\nabla p_i\|_{L^{2}(0,T;L^{2}(\Omega))},
\end{align*}
which, after division by $\tau$, tends to $0$ as $\tau \rightarrow 0$, due to Lemma \ref{IKP08_eq:2dot8}, Proposition \ref{prop_h2} and uniform boundedness of $\|\nabla \dot{u}_i^\tau\|_{L^{\infty}(0,T;L^{\infty}(\Omega))}$. It remains to estimate the first two lines in \eqref{iii}. We will first rewrite them using the representation formula \eqref{aa_formula} as
\begin{align*}
&\int_0^T \int_{\Omega_i}b_i \delta_i\Bigl(|A_\tau \nabla \dot{u}_i^\tau|^{q-1}A_\tau \nabla \dot{u}_i^\tau- |A_\tau \nabla \dot{u}_i|^{q-1}A_\tau \nabla \dot{u}_i\\
&-(|\nabla \dot{u}_i^\tau|^{q-1}\nabla \dot{u}_i^\tau-|\nabla \dot{u}_i|^{q-1} \nabla \dot{u}_i)\Bigr)\cdot A_\tau\nabla p_i I_\tau \, dx \, ds\\
=& \, \int_0^T \int_{\Omega_i}b_i \delta_i\Bigl\{(A_\tau-I)\nabla(\dot{u}_i^\tau-\dot{u}_i)\int_0^1|A_\tau \nabla \dot{u}_{i}+\sigma A_\tau \nabla (\dot{u}_{i}^\tau-\dot{u}_i)|^{q-1}\, d \sigma \cdot  A_\tau\nabla p_i I_\tau\\
&+(\nabla(\dot{u}_i^\tau-\dot{u}_i)\int_0^1\Bigl(|A_\tau \nabla \dot{u}_i+\sigma A_\tau \nabla (\dot{u}_i^\tau-\dot{u}_{i})|^{q-1}-| \nabla \dot{u}_{i}+\sigma  \nabla (\dot{u}_{i}^\tau-\dot{u}_{i})|^{q-1}\, \Bigr) d\sigma \\
&+ (q-1)\int_0^1 \Bigl(\mathcal{L}(A_\tau \nabla \dot{u}_i+\sigma A_\tau \nabla (\dot{u}_i^\tau-\dot{u}_{i}),A_\tau \nabla (\dot{u}_i^\tau-\dot{u}_i)) \\
&\quad -\mathcal{L}(\nabla \dot{u}_{i}+\sigma  \nabla (\dot{u}_{i}^\tau-\dot{u}_{i}),\nabla (\dot{u}_i^\tau-\dot{u}_i))\Bigr)\, d \sigma)\, \cdot  A_\tau\nabla p_i I_\tau\Bigr\} \, dx \, ds.
\end{align*}
with $\mathcal{L}$ as in \eqref{calL}.
For the first line on the right hand side (divided by $\tau$) we immediately have convergence to $0$, thanks to Lemma \ref{IKP08_eq:2dot8}, uniform boundedness of $A_\tau$ and Proposition \ref{prop_h2}. For the remaining terms we further have, due to inequalities \eqref{aa_ineq3} and \eqref{aa_ineq2}
\begin{align*}
&\Bigl|\int_0^T \int_{\Omega_i} \nabla(\dot{u}_i^\tau-\dot{u}_i)\int_0^1\Bigl(|A_\tau \nabla \dot{u}_i+\sigma A_\tau \nabla (\dot{u}_i^\tau-\dot{u}_{i})|^{q-1}\\
&-| \nabla \dot{u}_{i}+\sigma  \nabla (\dot{u}_{i}^\tau-\dot{u}_{i})|^{q-1}\, \Bigr) d\sigma \,\cdot  A_\tau\nabla p_i I_\tau  dx \, ds \Bigr| \\
\leq& \, C_q \|\nabla (\dot{u}_i^\tau-\dot{u}_i)\|_{L^{2}(0,T;L^{2}(\Omega_i))}|A_\tau-I|_{L^{\infty}(\Omega_i)}(\|\nabla \dot{u}_i\|_{L^{\infty}(0,T;L^{\infty}(\Omega_i))}\\
&+\|\nabla (\dot{u}_i^\tau-\dot{u}_i)\|_{L^{\infty}(0,T;L^{\infty}(\Omega_i))}) 
 \Bigl((1+\beta_1^{q-2})\|\nabla \dot{u}_i\|_{L^{\infty}(0,T;L^{\infty}(\Omega_i))}^{q-2}\\
 &+(1+\beta_1^{q-2})\|\nabla (\dot{u}_i^\tau-\dot{u}_i)\|_{L^{\infty}(0,T;L^{\infty}(\Omega_i))}^{q-2}\Bigr)\,\alpha_1 \beta_1\|\nabla p_i\|_{L^{2}(0,T;L^{2}(\Omega_i))}.
\end{align*}
Furthermore, by employing estimate \eqref{aa_ineq6} with $\eta=0$ we obtain
\begin{align*}
&\Bigl|(q-1)\int_0^1 (\mathcal{L}(A_\tau \nabla \dot{u}_i+\sigma A_\tau \nabla (\dot{u}_i^\tau-\dot{u}_{i}),A_\tau \nabla (\dot{u}_i^\tau-\dot{u}_i)) \\
& \quad-\mathcal{L}(\nabla \dot{u}_{i}+\sigma  \nabla (\dot{u}_{i}^\tau-\dot{u}_{i}),\nabla (\dot{u}^\tau_i-\dot{u}_i)))\, d \sigma\Bigr| \\
\leq & \, C_q |A_\tau-I|_{L^{\infty}(\Omega_i)}|\nabla(\dot{u}_i^\tau-\dot{u}_i)|\Bigl\{(|\nabla \dot{u}_i|+|\nabla (\dot{u}_i^\tau-\dot{u}_i)|)^2 ( \beta_1^{q-3}+1)(|\nabla \dot{u}_i|^{q-3}\\
&+| \nabla (\dot{u}_i^\tau-\dot{u}_i)|^{q-3})\beta_1^2
+(|\nabla \dot{u}_i|^{q-2}+|\nabla (\dot{u}_i^\tau-\dot{u}_i)|^{q-2})(1+\beta_1)(|\nabla \dot{u}_i|\\
&+|\nabla (\dot{u}_i^\tau-\dot{u}_i)|)\Bigr \},
\end{align*}
valid for $q >2$, from which, by making use of Proposition \ref{prop_h2}, Lemma \ref{IKP08_eq:2dot8} and hypothesis ($\mathcal{H}_1$), we finally have  $\frac{1}{\tau} III_i \rightarrow 0$ as $\tau \rightarrow 0$ and therefore Proposition \ref{prop_h4} holds.
\end{proof}
\noindent For the second property to hold we have to assume that $p$ is slightly more than $W^{1,2}$ regular on the subdomains, i.e.: \\

($\mathcal{H}_3$) \hspace{2mm} $p_{\vert \Omega_i} \in L^{2+\varepsilon}(0,T;W^{1,2+\varepsilon}(\Omega_i))$ for some $\varepsilon>0$.

\begin{proposition} \label{prop_h3}  Let $q >2$ and assumptions \eqref{coeff_3} on the coefficients hold. Assume that hypotheses \text{\normalfont ($\mathcal{H}_1$)-($\mathcal{H}_3$)} are valid. Then
\begin{align*} 
 \displaystyle \lim_{\tau \rightarrow 0} \frac{1}{\tau}\langle E(u^\tau,\Omega_+)-E(u,\Omega_+)-E_u(u,\Omega_+)(u^\tau-u),p \rangle_{\hat{X}^\star,\hat{X}}=0, 
\end{align*}
where $p$ is the adjoint state. 
\end{proposition}
\begin{proof}
\noindent We have 
\begin{align} \label{h3_1}
&\langle E(u^\tau,\Omega_+)-E(u,\Omega_+)-E_u(u,\Omega_+)(u^\tau-u),p \rangle_{\hat{X}^\star,\hat{X}} \nn \\
=&\displaystyle \sum_{i \in \{+,-\}} \int_0^T \int_{\Omega_i} -\frac{2k_i}{\lambda_i}\Bigl((u_i^\tau-u_i)(\ddot{u}_i^\tau-\ddot{u}_i)+(\dot{u}_i^\tau-\dot{u}_i)^2\Bigr)p_i\, dx \, ds \nn\\
&+\displaystyle \sum_{i \in \{+,-\}}\int_0^T \int_{\Omega} b_i\delta_i \Bigl(|\nabla \dot{u}_i^\tau|^{q-1}\nabla \dot{u}_i^\tau-|\nabla \dot{u}_i|^{q-1}\nabla \dot{u}_i-|\nabla \dot{u}_i|^{q-1}\nabla (\dot{u}_i^\tau-\dot{u}_i) \\
&-(q-1)|\nabla \dot{u}_i|^{q-3}(\nabla \dot{u}_i\cdot \nabla (\dot{u}_i^\tau-\dot{u}_i))\nabla \dot{u}_i\Bigr)\cdot \nabla p_i \, dx \, ds.\nn
\end{align}
The first sum on the right hand side can be estimated as follows
\begin{align*}
 &\Bigl|\displaystyle \sum_{i \in \{+,-\}} \int_0^T \int_{\Omega_i} -\frac{2k_i}{\lambda_i}\Bigl((u^\tau_i-u_i)(\ddot{u}_i^\tau-\ddot{u}_i)+(\dot{u}_i^\tau-\dot{u}_i)^2\Bigr)p_i\, dx \, ds \Bigr| \\
 =&\, \Bigl|\displaystyle \sum_{i \in \{+,-\}}\Bigl[\int_{\Omega_i}\frac{2k_i}{\lambda_i}(u^\tau_i-u_i)(\dot{u}^\tau_i-\dot{u}_i)\, p_i \,dx\Bigr]_0^T+\displaystyle \sum_{i \in \{+,-\}}\int_0^T \int_{\Omega} \frac{2k_i}{\lambda_i}(u_i^\tau-u_i)(\dot{u}_i^\tau-\dot{u}_i)\, \dot{p}_i \, dx\, ds\Bigr| \\
\leq& \, \displaystyle \sum_{i \in \{+,-\}}\frac{2 \overline{k}_i}{\underline{\lambda}_i} (C^{\Omega_i}_{H^1,L^4})^2 \|u_i^\tau-u_i\|_{L^{\infty}(0,T;H^1(\Omega_i))}\|\dot{u}_i^\tau-\dot{u}_i\|_{L^\infty(0,T;L^2(\Omega_i))}\|\dot{p}_i\|_{L^{2}(0,T;H^1(\Omega_i))},
\end{align*}
since $(u -u^\tau) \vert_{t=0}=(\dot{u} -\dot{u}^\tau) \vert_{t=0}=0$ and $p\vert_{t=T}=\dot{p}\vert_{t=T}=0$. This expression, upon division by $\tau$, tends to $0$ as $\tau \rightarrow 0$, due to Proposition \ref{prop_h2}. The second sum in \eqref{h3_1} can be rewritten with the help of formula \eqref{aa_formula} as given below
\begin{align*}
&\displaystyle \sum_{i \in \{+,-\}}\int_0^T \int_{\Omega_i} b_i\delta_i\Bigl(|\nabla \dot{u}_i^\tau|^{q-1}\nabla \dot{u}_i^\tau-|\nabla \dot{u}_i|^{q-1}\nabla \dot{u}_i-|\nabla \dot{u}_i|^{q-1}\nabla (\dot{u}^\tau_i-\dot{u}_i)\\
&-(q-1)|\nabla \dot{u}_i|^{q-3}(\nabla \dot{u}_i\cdot \nabla (\dot{u}_i^\tau-\dot{u})_i)\nabla \dot{u}_i\Bigr)\cdot \nabla p_i\, dx \, ds \\
=&\displaystyle \sum_{i \in \{+,-\}}\int_0^T \int_{\Omega_i} b_i\delta_i \, \int_0^1 (|\nabla \dot{u}_i+\sigma\nabla (\dot{u}_i^\tau-\dot{u}_i)|^{q-1}-|\nabla \dot{u}_i|^{q-1})\, d\sigma \,\nabla(\dot{u}_i^\tau-\dot{u}_i)\cdot \nabla p_i  \, dx \, ds \\
&+\displaystyle \sum_{i \in \{+,-\}}\int_0^T \int_{\Omega} b_i\delta_i(q-1)\int_0^1 \Bigl( \mathcal{L}(\nabla \dot{u}_i+\sigma\nabla (\dot{u}^\tau-\dot{u}_i),\nabla (\dot{u}_i^\tau-\dot{u}_i))\\
& -\mathcal{L}(\nabla \dot{u}_i,\nabla (\dot{u}_i^\tau-\dot{u}_i)) \Bigr) \cdot \nabla p_i \, d\sigma \, dx \, ds \\
:=& \, I+II,
\end{align*}
where $\mathcal{L}$ can be estimated as in \eqref{aa_ineq6}. By employing inequality \eqref{aa_ineq3} we obtain
\begin{align*}
|I|\leq &\displaystyle \sum_{i \in \{+,-\}}\int_0^T \int_{\Omega_i} \overline{b} \hspace{0.08em} \overline{\delta } C_q \, |\nabla(\dot{u}_i^\tau-\dot{u}_i)|^{2-\eta}\int_0^1 (|\nabla \dot{u}_i+\sigma\nabla (\dot{u}_i^\tau-\dot{u}_i)|+|\nabla \dot{u}_i|)^{q-2+\eta} \, d\sigma \, | \nabla p_i|   \, dx \, ds.
\end{align*}
Making use of H\" older's inequality and hypothesis $(\mathcal{H}_1)$ results in
\begin{align*}
|I| \leq & \, \overline{b} \hspace{0.08em} \overline{\delta} C_q \sum_{i \in \{+,-\}} \int_0^T \int_{\Omega} |\nabla(\dot{u}_i^\tau-\dot{u}_i)|^{2-\eta}(  |\nabla \dot{u}_i|^{q-2+\eta}+|\nabla(\dot{u}^\tau_i-\dot{u}_i)|^{q-2+\eta}) |\nabla p_i| \, dx \, ds  \nonumber \\
\leq& \, \overline{b} \hspace{0.08em} \overline{\delta} C_q\, \sum_{i \in \{+,-\}}\|\nabla (\dot{u}_i^\tau-\dot{u}_i)\|^{2-\eta}_{L^{2}(0,T;L^{2}(\Omega))}(\|\nabla \dot{u}_i\|^{q-2+\eta}_{L^{\infty}(0,T;L^{\infty}(\Omega_i))}\\
&+\|\nabla (\dot{u}_i^\tau-\dot{u}_i)\|^{q-2+\eta}_{L^{\infty}(0,T;L^{\infty}(\Omega))})\|\nabla p_i\|_{L^{\frac{2}{\eta}}(0,T;L^{\frac{2}{\eta}}(\Omega))} \nonumber.
\end{align*}
Here we can choose $\eta=\frac{2}{2+\varepsilon}$. Recall that $\frac{1}{\tau}\|\nabla (\dot{u}_i^\tau-\dot{u}_i)\|_{L^{2}(0,T;L^{2}(\Omega))}$ is uniformly bounded for $|\tau|<\tau_0$, $\tau \neq 0$, due to Proposition \ref{prop:un_bound} and that thanks to hypothesis $(\mathcal{H}_2)$ we have uniform boundedness of $\|\nabla \dot{u}_i^\tau\|_{L^{\infty}(0,T;L^{\infty}(\Omega))}$ as well. This means that we can achieve that $I$ upon division by $\tau$ tends to zero as $\tau \rightarrow 0$. \\
\indent By employing inequality \eqref{aa_ineq6}, we get the estimate for $II$:
\begin{align*} 
|II| \leq & \,  \overline{b} \hspace{0.08em} \overline{\delta}C_q \,\sum_{i \in \{+,-\}} \int_0^T \int_{\Omega_i} |\nabla (\dot{u}_i^\tau-\dot{u})|^{2-\eta}\Bigl\{|\nabla \dot{u}_i|^{q-3+\eta}+|\nabla (\dot{u}^\tau_i-\dot{u}_i)|^{q-3+\eta})(|\nabla \dot{u}_i|+|\nabla \dot{u}_i^\tau|)\\
&+|\nabla\dot{u}_i|^{q-2}(|\nabla \dot{u}_i|^\eta+|\nabla \dot{u}^\tau_i|^\eta)\Bigr\}\,|\nabla p_i| \, dx \, ds \\
\leq& \,\overline{b}\hspace{0.08em} \overline{\delta}C_q \, \sum_{i \in \{+,-\}}\|\nabla (\dot{u}^\tau-\dot{u}_i)\|^{2-\eta}_{L^{2}(0,T;L^{2}(\Omega_i))}\Bigl\{(\|\nabla \dot{u}_i\|^{q-3+\eta}_{L^{\infty}(0,T;L^{\infty}(\Omega_i))} \nonumber \\
&+\|\nabla (\dot{u}_i^\tau-\dot{u}_i)\|^{q-3+\eta}_{L^{\infty}(0,T;L^{\infty}(\Omega_i))})(\|\nabla \dot{u}_i\|_{L^{\infty}(0,T;L^{\infty}(\Omega_i))}
+\|\nabla \dot{u}_i^\tau\|_{L^{\infty}(0,T;L^{\infty}(\Omega_i))})\\
&+\|\nabla \dot{u}\|^{q-2}_{L^{\infty}(0,T;L^{\infty}(\Omega))}(\|\nabla \dot{u}\|^\eta_{L^{\infty}(0,T;L^{\infty}(\Omega_i))}
+\|\nabla \dot{u}_i^\tau\|^\eta_{L^{\infty}(0,T;L^{\infty}(\Omega_i))})\Bigr\}\|\nabla p_i\|_{L^{\frac{2}{\eta}}(0,T;L^{\frac{2}{\eta}}(\Omega_i))}, \nonumber
\end{align*}
with $\eta=\frac{2}{2+\varepsilon}$. Upon division by $\tau$, due to Propositions \ref{prop_h2} and \ref{prop:un_bound}, the right hand side tends to zero as $\tau \rightarrow 0$. 
\end{proof}
\section{Computation of the shape derivative} \label{computation_shape}
\vspace{2mm}
\indent Let $u^\tau$, $u$ satisfy $\tilde{E}(u^\tau,\tau)=0$ and $E(u,\Omega_+)=0$, for $|\tau|<\tau_0$, $\tau \in \mathbb{R}$. In  that case $u_{\tau}=u^\tau \circ F_\tau$ is the solution of $E(u_{\tau},\Omega_{+,\tau})=0$. We then have 
\begin{align*}
dJ(u,\Omega_+)h=& \,\lim_{\tau \rightarrow 0} \frac{1}{\tau} \, \int_0^T \int_{\Omega}(j(u^\tau)I_\tau-j(u)) \, dx \, ds\\
=& \, \int_0^T \int_{\Omega}(j^\prime(u) \displaystyle \lim_{\tau \rightarrow 0} \frac{u^\tau-u}{\tau}+j(u)\div h) \, dx \, ds,
\end{align*}
where we have used (similarly to Lemma 2.1, \cite{IKP08}) that
\begin{align*}
&\displaystyle \lim_{\tau \rightarrow 0} \frac{1}{\tau} \Bigl|\int_0^T \int_{\Omega}\Bigl(j(u^\tau)-j(u)-j^\prime(u)(u^\tau-u)\Bigr)I_\tau\, dx \, ds\Bigr| \\
\leq & \, \displaystyle \lim_{\tau \rightarrow 0} \frac{1}{\tau} \Bigl|\int_0^T \int_{\Omega} (u^\tau-u)^2I_\tau \, dx \, ds\Bigr|=0,
\end{align*}
which follows from Proposition \ref{prop_h2} and the fact that $I_\tau$ is uniformly bounded for $\tau \in (-\tau_0,\tau_0)$. By employing the adjoint problem \eqref{adjoint_weakform} and then proceeding as in the proof of Theorem 2.1, \cite{IKP08}, we obtain
\begin{align*}
\int_0^T \int_{\Omega} j^\prime(u)(u^\tau-u) \, dx \, ds 
=& \, \langle E_u(u,\Omega_+)(u^\tau-u),p\rangle_{\tilde{X}^\star,\tilde{X}} \\
=& \, -\langle E(u^\tau,\Omega_+)-E(u,\Omega_+)-E_u(u,\Omega_+)(u^\tau-u),p\rangle_{\tilde{X}^\star,\tilde{X}} \\
&-\langle \underbrace{\tilde{E}(u^\tau,\tau)}_{=0}-\tilde{E}(u,\tau)-(E(u^\tau,\Omega_+)-E(u,\Omega_+)),p\rangle_{\tilde{X}^{\star},\tilde{X}} \\
&-\langle \tilde{E}(u,\tau) -\underbrace{\tilde{E}(u,0)}_{=0},p\rangle_{\tilde{X}^\star,\tilde{X}}.
\end{align*}
The second and third line divided by $\tau$ tend to zero, as $\tau \rightarrow 0$, on account of Propositions \ref{prop_h4} and \ref{prop_h3} and we are left with
\begin{align*}
&\displaystyle \lim_{\tau \rightarrow 0} \frac{1}{\tau} \int_0^T \int_{\Omega} j^\prime(u)(u^\tau-u) \, dx \, ds \\
=&- \, \displaystyle \lim_{\tau \rightarrow 0} \frac{1}{\tau}\langle \tilde{E}(u,\tau) -\tilde{E}(u,0),p\rangle_{\tilde{X}^\star,\tilde{X}}.
\end{align*}
\noindent This limit, representing the (artificial) time derivative of $\langle \tilde{E}(u,\tau),p \rangle_{\tilde{X}^{\star},\tilde{X}}$, is typically computed by transforming the expressions $\tilde{E}(u,\tau)$ and $p$ back to $E(u \circ F_\tau^{-1}, \Omega_{+,\tau})$ and $p \circ F_\tau^{-1}$, and then making use of differentiation rules for $u \circ F_\tau^{-1}$ and $p \circ F_\tau^{-1}$ (see Examples 1-5 and Lemma 2.4, \cite{IKP08}). However, these rules assume $H^2$ differentiability in space of the primal and the adjoint state, which is too high of a requirement in our case. Instead, we continue with calculating the difference
\begin{align*}
&\langle \tilde{E}(u,\tau) -\tilde{E}(u,0),p\rangle_{\tilde{X}^\star,\tilde{X}} = I_++I_-,
\end{align*}
where the two terms on the right hand side are given by
\begin{align} \label{weak_shape1}
I_i=& \,\int_0^T \int_{\Omega_i} \Bigl\{\frac{1}{\lambda_i}(1-2k_iu_i)\ddot{u}_i p_i(I_\tau-1) \nn \\
&+\frac{1}{\varrho_i}((A_\tau-I)\nabla u_i \cdot \nabla p_i+A_\tau \nabla u_i \cdot (A_\tau-I)\nabla p_i+A_\tau \nabla u_i\cdot A_\tau \nabla p_i(I_\tau-1))\nn \\
&+b_i(1-\delta_i)((A_\tau-I)\nabla \dot{u}_i \cdot \nabla p_i+A_\tau \nabla \dot{u}_i \cdot (A_\tau-I)\nabla p_i
+A_\tau \nabla \dot{u}_i\cdot A_\tau \nabla p_i(I_\tau-1)) \nn\\
&+b_i\delta_i\Bigl((A_\tau-I)\nabla \dot{u}_i \int_0^1 |\nabla \dot{u}_i+\sigma(A_\tau-I)\nabla \dot{u}_i|^{q-1}\, d\sigma \, \cdot A_\tau \nabla p_i I_\tau\\
&+(q-1)\int_0^1|\nabla \dot{u}_i+\sigma(A_\tau-I)\nabla \dot{u}_i|^{q-3}(\nabla \dot{u}_i+\sigma(A_\tau-I)\nabla \dot{u}_i) \nn\\
&\cdot (A_\tau-I)\nabla \dot{u}_i(\nabla \dot{u}_i+\sigma(A_\tau-I)\nabla \dot{u}_i)\, d\sigma \, \cdot A_\tau \nabla p_i I_\tau \nn\\
&+|\nabla \dot{u}_i|^{q-1}\nabla \dot{u}_i \cdot ((A_\tau-I)\nabla p_i I_\tau+(I_\tau-1)\nabla p_i)\Bigr)
-\frac{2k_i}{\lambda_i}(\dot{u}_i)^2p_i(I_\tau-1)\Bigr\} \, dx \, ds,\nn
\end{align}
$i \in \{+,-\}$, and we have employed the formula \eqref{aa_formula} to represent the difference 
$|A_\tau\nabla \dot{u}_i|^{q-1}A_\tau\nabla \dot{u}_i-|\nabla \dot{u}_i|^{q-1}\nabla \dot{u}_i$. Dividing \eqref{weak_shape1} by $\tau$, passing to the limit and utilizing Lemma  \ref{IKP08_eq:2dot8} yields
\begin{align} \label{aa_shapecomp_diff}
&\displaystyle \lim_{\tau \rightarrow 0} \frac{1}{\tau} \, \langle \tilde{E}(u,\tau) -\tilde{E}(u,0),p\rangle_{\tilde{X}^\star,\tilde{X}} \nn  \\
=& \, \displaystyle \sum_{i \in \{+,-\}}\int_0^T \int_{\Omega_i}\Bigl\{ 
-(\frac{1}{\varrho_i} \nabla u_i^T+b_i(1-\delta_i) \nabla \dot{u}_i^T 
+b_i\delta_i |\nabla \dot{u}_i |^{q-1} \nabla \dot{u}_i^T ) (Dh^T \nabla p_i+Dh \nabla p_i)    \nn \\
& +b_i\delta_i(q-1)|\nabla \dot{u}_i|^{q-3}(\nabla \dot{u}_i \cdot (-Dh)^T \nabla \dot{u}_i )( \nabla \dot{u}_i \cdot \nabla p_i ) \Bigr\} \, dx \, ds  \\
&+\displaystyle \sum_{i \in \{+,-\}} \int_0^T \int_{\Omega_i} \Bigl\{\frac{1}{\lambda_i}(1-2k_i  u_i )\ddot{u}_i p_i  + \frac{1}{\varrho_i}\nabla u_i \cdot \nabla p_i +b_i(1-\delta_i)\nabla \dot{u}_i \cdot \nabla p_i\nn\\
& +b_i\delta_i |\nabla \dot{u}_i |^{q-1}\nabla \dot{u}_i \cdot \nabla p_i-\frac{2k_i}{\lambda_i}(\dot{u}_i)^2p_i \Bigr\} \,\div h \, dx \, ds, \nn
\end{align}
and we can now express the Eulerian derivative:
\begin{theorem}(Weak shape derivative) \label{weak_shape} Let $q > 2 $, $u_0, u_1 \in W_0^{1,q+1}(\Omega)$, and assumptions \eqref{coeff_3} on coefficients hold. Assume that \text{\normalfont ($\mathcal{H}_1$)}-\text{\normalfont ($\mathcal{H}_3$) } are valid. Then the shape derivative of $J$ at $\Omega_+$ with respect to $h \in C^{1,1}(\bar{\Omega},\mathbb{R}^d)$ can be represented as
\begin{align} \label{strong}
dJ(u,\Omega_+)h  
=& \, \int_0^T \int_{\Omega}\Bigl\{ 
(\frac{1}{\varrho} \nabla u^{\tau}+b(1-\delta) \nabla \dot{u}^T+b\delta |\nabla \dot{u} |^{q-1} \nabla \dot{u}^T )(Dh^T \nabla p+Dh \nabla p)   \nn \\
& +b\delta(q-1)|\nabla \dot{u}|^{q-3}(\nabla \dot{u} \cdot Dh^T \nabla \dot{u})( \nabla \dot{u} \cdot \nabla p ) \Bigr\} \, dx \, ds \nn \\
&-\int_0^T \int_{\Omega} \Bigl\{\frac{1}{\lambda}(1-2ku )\ddot{u} p  + \frac{1}{\varrho}\nabla u \cdot \nabla p +b(1-\delta)\nabla \dot{u} \cdot \nabla p \\
& +b\delta|\nabla \dot{u} |^{q-1}\nabla \dot{u}\cdot \nabla p-\frac{2k}{\lambda}(\dot{u})^2p -j(u)\Bigr\} \,\div h \, dx \, ds.\nn
\end{align}
\end{theorem}
Note that the integrals in \eqref{strong} are well-defined thanks to hypothesis $(\mathcal{H}_1)$, and for them to be well-defined hypothesis $(\mathcal{H}_3)$ is actually not necessary.\\
\indent Theorem \ref{weak_shape} gives us the shape derivative of the cost functional in terms of the volume integrals, which is in \cite{Berg} regarded as a weak shape derivative. However, an obvious advantage of the volume expression of the shape derivative is that it allows for a lower regularity of shapes as well as the lower regularity of the primal and the adjoint state. Recently there have been suggestions that the domain representation is also advantageous in terms of easiness of computation and numerical implementations (see, for example, \cite{Hiptmair}, \cite{Laurain}), especially in the case of transmission problems where shape derivatives given in terms of the boundary integrals contain jumps of functions over the interfaces, which is numerically a delicate task to perform. 
\subsection{Strong shape derivative} \label{strong_shape}
 In order to express the shape derivative in the form required by the Delfour-Hadamard-Zol\' esio structure theorem we would have to apply Green's theorem to the last two lines in \eqref{aa_shapecomp_diff}, which is not allowed since $u$ and $p$ are not sufficiently regular. However, it turns out that if the domains are sufficiently smooth and $\nabla \dot{u}$ is bounded in $L^{\infty}(0,T;L^\infty(\Omega))$, the state variable exhibits $H^2$ regularity on each of the subdomains. This result together with an assumption regarding the regularity of the trace of $\nabla u_{\vert_{\Omega_i}}$ and $\nabla p_{\vert_{\Omega_i}}$ on $\Gamma$ makes expressing the shape derivative of the cost functional in terms of the boundary integrals possible. Let $\partial \Omega$ and $\Gamma=\partial \Omega_+$ be $C^{1,1}$ regular. We utilize the following result (cf. Theorem 2, \cite{KN}):
\begin{theorem} \label{thm:higher_reg}
Assume that $q\geq 1$, $q>d-1$, $u_0\vert_{\Omega_i} \in H^2(\Omega_i)$, $u_0,u_1 \in W_0^{1,q+1}(\Omega)$, and that $\partial \Omega$ and $\Gamma=\partial \Omega_+$ are $C^{1,1}$ regular. Let $u$ be the weak solution of \eqref{state_problem}.  If $u \in W^{1,\infty}(0,T;W^{1,\infty}(\Omega))$ and $\|\nabla \dot{u}\|_{L^{\infty}(0,T;L^{\infty}(\Omega))}$ is sufficiently small, then $u_i \in H^1(0,T;H^2(\Omega_i))$, $i \in \{+,-\}$.
\end{theorem}

\indent To be able to express the shape derivative in terms of the boundary integrals over $\Gamma$, we first impose additional regularity hypotheses on $u$ and $p$:\\

$(\mathcal{H}_4)$ \hspace{2mm}  $\text{tr}^{\Omega_i}_{\Gamma} \nabla p \in L^1(0,T;L^{1}(\Gamma))$, $i \in \{+,-\}$,\\

$(\mathcal{H}_5)$ \hspace{2mm}  $\text{tr}^{\Omega_i}_{\Gamma} \nabla u \in L^\infty(0,T;L^\infty(\Gamma))$, $i \in \{+,-\}$,\\
 
\noindent where $\text{tr}^{\Omega_i}_{\Gamma}\nabla u$ and $\text{tr}^{\Omega_i}_{\Gamma}\nabla p$ stand for the trace of $\nabla u \vert_{\Omega_i}$ and $\nabla p \vert_{\Omega_i}$, respectively, on $\Gamma$. Hypotheses $(\mathcal{H}_4)$ and $(\mathcal{H}_5)$ will ensure that the forthcoming boundary integrals are well-defined. 
Note that they do not follow from the previous hypotheses $(\mathcal{H}_1)$-$(\mathcal{H}_3)$ and regularity results, partially due to the lack of an appropriate trace theorem in the limiting $L^\infty$ case.
\\
\indent Next, we introduce sufficiently smooth in space approximations of the adjoint state in $H^1(0,T;H^1(\Omega_i))$. Fix $i \in \{+,-\}$. Let $\{p_{i,m}\}_{m=1}^\infty \subset H^1(0,T;C^{\infty}(\Omega_i))$ be a sequence that converges to $p_i$ in $H^1(0,T;H^1(\Omega_i))$ and such that $p_{i,m}=p_i$ on $\partial \Omega_i$ (cf. Theorem 3.42, \cite{Demengel}). 
\indent We can then approximate \eqref{aa_shapecomp_diff} as 
\begin{align} \label{aa_shapecomp_diff_approx}
&\displaystyle \lim_{\tau \rightarrow 0} \frac{1}{\tau} \, \langle \tilde{E}(u,\tau) -\tilde{E}(u,0),p\rangle_{\tilde{X}^\star,\tilde{X}} \nn  \\
=& \, \displaystyle \sum_{i \in \{+,-\}}\int_0^T \int_{\Omega_i}\Bigl\{ 
-(\frac{1}{\varrho_i} \nabla u_i^T+b_i(1-\delta_i) \nabla \dot{u}_i^T 
+b_i\delta_i |\nabla \dot{u}_i |^{q-1} \nabla \dot{u}_i^T ) (Dh^T \nabla p_{i,m}\nn \\
&+Dh \nabla p_{i,m})   
+b_i\delta_i(q-1)|\nabla \dot{u}_i|^{q-3}(\nabla \dot{u}_i \cdot (-Dh)^T \nabla \dot{u}_i )( \nabla \dot{u}_i \cdot \nabla p_{i,m} ) \Bigr\} \, dx \, ds  \\
&+\displaystyle \sum_{i \in \{+,-\}} \int_0^T \int_{\Omega_i} \Bigl\{\frac{1}{\lambda_i}(1-2k_i  u_i )\ddot{u}_i p_{i,m}  + \frac{1}{\varrho_i}\nabla u_i \cdot \nabla p_{i,m}+b_i(1-\delta_i)\nabla \dot{u}_i \cdot \nabla p_{i,m} \nn\\
& +b_i\delta_i |\nabla \dot{u}_i |^{q-1}\nabla \dot{u}_i \cdot \nabla p_{i,m} 
-\frac{2k_i}{\lambda_i}(\dot{u}_i)^2p_{i,m} \Bigr\} \,\div h \, dx \, ds+R_1(p_i,p_{i,m}) \nn,
\end{align}
where the error term is given by
\begin{align*}
R_1(p_i,p_{i,m})=& \,\displaystyle \sum_{i \in \{+,-\}}\int_0^T \int_{\Omega_i}\Bigl\{ 
-(\frac{1}{\varrho_i} \nabla u_i^T+b_i(1-\delta_i) \nabla \dot{u}_i^T \nn \\
&+b_i\delta_i |\nabla \dot{u}_i |^{q-1} \nabla \dot{u}_i^T ) (Dh^T \nabla (p_i-p_{i,m})+Dh \nabla (p_i-p_{i,m}))    \nn \\
& +b_i\delta_i(q-1)|\nabla \dot{u}_i|^{q-3}(\nabla \dot{u}_i \cdot (-Dh)^T \nabla \dot{u}_i )( \nabla \dot{u}_i \cdot \nabla (p_i-p_{i,m}) ) \Bigr\} \, dx \, ds  \\
&+\displaystyle \sum_{i \in \{+,-\}} \int_0^T \int_{\Omega_i} \Bigl\{\frac{1}{\lambda_i}(1-2k_i  u_i )\ddot{u}_i (p_i-p_{i,m})  + \frac{1}{\varrho_i}\nabla u_i \cdot \nabla (p_i-p_{i,m}) \nn\\
&+b_i(1-\delta_i)\nabla \dot{u}_i \cdot \nabla (p_i-p_{i,m}) +b_i\delta_i |\nabla \dot{u}_i |^{q-1}\nabla \dot{u}_i \cdot \nabla (p_i-p_{i,m})\\
&-\frac{2k_i}{\lambda_i}(\dot{u}_i)^2(p_i-p_{i,m}) \Bigr\} \,\div h \, dx \, ds.
\end{align*}
\noindent Since $u_{i}$ and $ p_{i,m}$ are sufficiently smooth, we are allowed to employ Green's theorem in \eqref{aa_shapecomp_diff_approx}. This will cause the terms containing $Dh$ (not included in $R_1$) to cancel out, and we arrive at
\begin{align*}
&\displaystyle \lim_{\tau \rightarrow 0} \frac{1}{\tau} \, \langle \tilde{E}(u,\tau) -\tilde{E}(u,0),p\rangle_{\tilde{X}^\star,\tilde{X}}  \nn\\
=& \, \displaystyle \sum_{i \in \{+,-\}} \int_0^T \int_{\partial \Omega_i} \Bigl\{\frac{1}{\lambda_i}(1-2k_iu_{i})\ddot{u}_{i}p_{i,m}+\frac{1}{\varrho_i}\nabla u_{i} \cdot \nabla p_{i,m}  \nn  \\
&+b_i(1-\delta_i)\nabla \dot{u}_{i} \cdot \nabla p_{i,m} +b_i\delta_i |\nabla \dot{u}_{i}|^{q-1}\nabla \dot{u}_{i} \cdot \nabla p_{i,m}
-\frac{2k_i}{\lambda_i}(\dot{u}_{i})^2p_{i,m}\Bigr\}\, h^T n_i \, dx \, ds  \\
&-\displaystyle \sum_{i \in \{+,-\}}\int_0^T \int_{\Omega_i}\Bigl\{\frac{1}{\lambda_i}(1-2k_iu_i)\ddot{u}_i(\nabla p_{i,m}^Th)+\frac{1}{\varrho_i} \nabla u_i \cdot \nabla (\nabla p_{i,m}^Th)\nn \\
&+b_i(1-\delta_i)\nabla \dot{u}_i \cdot \nabla (\nabla p_{i,m}^Th)+b_i\delta_i|\nabla \dot{u}_{i}|^{q-1}\nabla \dot{u}_i \cdot \nabla (\nabla p_{i,m}^Th)\Bigr\}  \, dx \, ds  \\
&-\displaystyle \sum_{i \in \{+,-\}} \int_0^T \int_{\Omega_i}\Bigl\{\frac{1}{\lambda_i}(1-2k_iu_i)\ddot {p}_{m,i}(\nabla u_i^Th)+\frac{1}{\varrho_i} \nabla p_{i,m} \cdot \nabla (\nabla u_i^Th) \nn \\
&-b_i(1-\delta_i)\nabla \dot{p}_{m,i} \cdot \nabla (\nabla u_i^Th)
-b_i\delta_i \,(G_{u_i}(\nabla p_{i,m}))^{\Lcdot}  \cdot \nabla (\nabla u_i^Th) \Bigr\} \, dx \, ds +R_1(p_i,p_{i,m}). \nn 
\end{align*}
This expression can be rewritten as 
\begin{align*}
&\displaystyle \lim_{\tau \rightarrow 0} \frac{1}{\tau} \, \langle \tilde{E}(u,\tau) -\tilde{E}(u,0),p\rangle_{\tilde{X}^\star,\tilde{X}}  \nn\\
=& \, \displaystyle \sum_{i \in \{+,-\}} \int_0^T \int_{\partial \Omega_i} \Bigl\{\frac{1}{\lambda_i}(1-2k_iu_{i})\ddot{u}_{i}p_{i}+\frac{1}{\varrho_i}\nabla u_{i} \cdot \nabla p_{i}  \nn  \\
&+b_i(1-\delta_i)\nabla \dot{u}_{i} \cdot \nabla p_{i} +b_i\delta_i |\nabla \dot{u}_{i}|^{q-1}\nabla \dot{u}_{i} \cdot \nabla p_{i}
-\frac{2k_i}{\lambda_i}(\dot{u}_{i})^2p_{i}\Bigr\}\, h^T n_i \, dx \, ds  \\
&-\displaystyle \sum_{i \in \{+,-\}}\int_0^T \int_{\Omega_i}\Bigl\{\frac{1}{\lambda_i}(1-2k_iu_i)\ddot{u}_i(\nabla p_{i,m}^Th)+\frac{1}{\varrho_i} \nabla u_i \cdot \nabla (\nabla p_{i,m}^Th)\nn \\
&+b_i(1-\delta_i)\nabla \dot{u}_i \cdot \nabla (\nabla p_{i,m}^Th)+b_i\delta_i|\nabla \dot{u}_{i}|^{q-1}\nabla \dot{u}_i \cdot \nabla (\nabla p_{i,m}^Th)\Bigr\}  \, dx \, ds  \\
&-\displaystyle \sum_{i \in \{+,-\}} \int_0^T \int_{\Omega_i}\Bigl\{\frac{1}{\lambda_i}(1-2k_iu_i)\ddot {p}_{i}(\nabla u_i^Th)+\frac{1}{\varrho_i} \nabla p_{i} \cdot \nabla (\nabla u_i^Th) \nn \\
&-b_i(1-\delta_i)\nabla \dot{p}_{i} \cdot \nabla (\nabla u_i^Th)
-b_i\delta_i \,(G_{u_i}(\nabla p_{i}))^{\Lcdot}  \cdot \nabla (\nabla u_i^Th) \Bigr\} \, dx \, ds\nn \\
& +R_1(p_i,p_{i,m})+R_2(p_i,p_{i,m})+R_3(p_i,p_{i,m}), \nn 
\end{align*}
(the first and the third sum are written in terms of $p_i$ plus the error $R_2+R_3$) where the approximation error terms are given by 
\begin{align*}
R_2(p_i,p_{i,m})=& \, \displaystyle \sum_{i \in \{+,-\}} \int_0^T \int_{\Omega_i}\Bigl\{\frac{1}{\lambda_i}(1-2k_iu_i) (\ddot{p}_i-\ddot{p}_{m,i})(\nabla u_i^Th)\\
&+\frac{1}{\varrho_i} \nabla (p_i-{p}_{m,i}) \cdot \nabla (\nabla u_i^Th)
-b_i(1-\delta_i)\nabla (\dot{p}_i-\dot{p}_{m,i}) \cdot \nabla (\nabla u_i^Th) \nn \\
&-b_i\delta_i \,(G_{u_i}(\nabla (p_i-p_{i,m})))^{\Lcdot}  \cdot \nabla (\nabla u_i^Th) \Bigr\} \, dx \, ds,
\end{align*}
\begin{align*}
R_3(p_i,p_{i,m})=& \,-\displaystyle \sum_{i \in \{+,-\}} \int_0^T \int_{\partial \Omega_i} \Bigl\{\frac{1}{\lambda_i}(1-2k_iu_{i})\ddot{u}_{i}(p_i-p_{i,m})+\frac{1}{\varrho_i}\nabla u_{i} \cdot \nabla (p_i- p_{i,m})  \nn  \\
&+b_i(1-\delta_i)\nabla \dot{u}_{i} \cdot \nabla (p_i-p_{i,m}) +b_i\delta_i |\nabla \dot{u}_{i}|^{q-1}\nabla \dot{u}_{i} \cdot \nabla (p_i-p_{i,m})\\
&-\frac{2k_i}{\lambda_i}(\dot{u}_{i})^2(p_{i}-p_{i,m})\Bigr\}\, h^T n_i \, dx \, ds.
\end{align*}
\begin{remark}
If $u \in W^{1,\infty}(0,T;W^{1,\infty}(\Omega))$, we are allowed to use $\phi \in L^2(0,T;H^1_0(\Omega))$ as a test function in \eqref{ModWest}. To see this, let $\phi_{j} \in L^2(0,T;C_0^\infty(\Omega))$, $\phi_j \rightarrow \phi$ in $L^2(0,T;H^1_0(\Omega))$ as $j \rightarrow \infty$. Then
\begin{align*}
&\int_0^T \int_{\Omega} \{\frac{1}{\lambda}(1-2ku)\ddot{u} \phi+\frac{1}{\varrho}\nabla u \cdot \nabla \phi + b(1-\delta)\nabla \dot{u} \cdot \nabla \phi \\
&+b \delta|\nabla \dot{u}|^{q-1} \nabla \dot{u} \cdot \nabla \phi-\frac{2k}{\lambda}(\dot{u})^2 \phi\} \, dx \, ds \\
=& \, \int_0^T \int_{\Omega} \{\frac{1}{\lambda}(1-2ku)\ddot{u} (\phi-\phi_j)+\frac{1}{\varrho}\nabla u \cdot \nabla (\phi-\phi_j) + b(1-\delta)\nabla \dot{u} \cdot \nabla (\phi-\phi_j) \\
&+b \delta|\nabla \dot{u}|^{q-1} \nabla \dot{u} \cdot \nabla (\phi-\phi_j)-\frac{2k}{\lambda}(\dot{u})^2 (\phi-\phi_j)\} \, dx \, ds \\
\leq& \, \frac{1}{\underline{\lambda}}((1+a_0)\|\ddot u\|_{L^2(0,T;L^2(\Omega))}+2\overline{k}(C^\Omega_{H_0^1,L^4})^2\sqrt{T}\|u\|^2_{L^\infty(0,T;H_0^1(\Omega))})\|\phi-\phi_j\|_{L^2(0,T;L^2(\Omega))} \\
&+(\frac{1}{\underline{\varrho}}\|\nabla u\|_{L^2(0,T;L^2(\Omega))}+\overline{b}(1-\underline{\delta})\|\nabla \dot{u}\|_{L^2(0,T;L^2(\Omega))}\\
&+ \overline{b}\hspace{0.08em}\overline{\delta}\|\nabla \dot{u}\|^{q-1}_{L^\infty(0,T;L^\infty(\Omega))}\|\nabla \dot{u}\|_{L^2(0,T;L^2(\Omega))})\|\nabla (\phi-\phi_j)\|_{L^2(0,T;L^2(\Omega))} \\
\rightarrow & \, 0, \ \ \ \text{as} \ \ j \rightarrow \infty.
\end{align*}
\end{remark}
\noindent Note that $\nabla u_{i} ^Th \in L^2(0,T;H^{1}(\Omega_i))$ and $\nabla p_{i,m}^Th \in L^2(0,T;H^1(\Omega_i))$, however functions
\begin{align*}
\tilde{\phi}(x,t):=\begin{cases} \nabla u_+^T(x,t)h(x) \quad x \in \Omega_+ \\ \nabla u_-^T(x,t)h(x) \quad x \in  \Omega_-\end{cases} \quad \tilde{\zeta}(x,t):=\begin{cases} \nabla p_{m,+}^T(x,t)h(x) \quad x \in  \Omega_+ \\ \nabla p_{m,-}^T(x,t)h(x) \quad x \in  \Omega_-\end{cases}
\end{align*} do not have to be continuous across the boundary $\Gamma$ and we cannot use them directly as test functions in the weak formulations of the state and the adjoint problem. 
 We can instead employ the two-domain weak formulations 
 which results in
\begin{align*}
&\displaystyle \lim_{\tau \rightarrow 0} \frac{1}{\tau} \, \langle \tilde{E}(u,\tau) -\tilde{E}(u,0),p\rangle_{\tilde{X}^\star,\tilde{X}}  \nn\\
=& \, \displaystyle \sum_{i \in \{+,-\}} \int_0^T \int_{\partial \Omega_i} \Bigl\{\frac{1}{\lambda_i}(1-2k_iu_{i})\ddot{u}_{i}p_{i}+\frac{1}{\varrho_i}\nabla u_{i} \cdot \nabla p_{i}  \nn  \\
&+b_i(1-\delta_i)\nabla \dot{u}_{i} \cdot \nabla p_{i} +b_i\delta_i |\nabla \dot{u}_{i}|^{q-1}\nabla \dot{u}_{i} \cdot \nabla p_{i}
-\frac{2k_i}{\lambda_i}(\dot{u}_{i})^2p_{i}\Bigr\}\, h^T n_i \, dx \, ds  \\
&-\displaystyle \sum_{i \in \{+,-\}}\int_0^T \int_{\partial \Omega_i}\Bigl\{\frac{1}{\varrho_i} \frac{\partial u_{i}}{\partial n_{i}} 
+b_i(1-\delta_i)\frac{\partial \dot{u}_{i}}{\partial n_{i}}
+b_i\delta_i|\nabla \dot{u}_{i}|^{q-1}\frac{\partial \dot{u}_{i}}{\partial n_{i}}\Bigr\} \, (\nabla p_{i,m}^Th) \, dx \, ds  \\
& - \displaystyle \sum_{i \in \{+,-\}} \int_0^T \int_{\partial \Omega_i}\Bigl\{\frac{1}{\varrho_i} \frac{\partial p_{i}}{\partial n_{i}}-b_i(1-\delta_i)\frac{\partial \dot{p}_{i}}{\partial n_{i}}
-b_i\delta_i \,(G_{u_i}(\nabla p_i) \cdot n_i)^{\Lcdot} \Bigr\} \, (\nabla u_{i}^Th) \, dx \, ds \\&-\displaystyle \sum_{i \in \{+,-\}} \int_0^T \int_{\Omega_i} j^{\prime}(u_i) (\nabla u_{i}^Th) \, dx \, ds +R_1(p_i,p_{i,m})+R_2(p_i,p_{i,m})+R_3(p_i,p_{i,m}). \nn 
\end{align*}
Finally, this can be rewritten as
\begin{align*}
&\displaystyle \lim_{\tau \rightarrow 0} \frac{1}{\tau} \, \langle \tilde{E}(u,\tau) -\tilde{E}(u,0),p\rangle_{\tilde{X}^\star,\tilde{X}}  \nn\\
=& \, \displaystyle \sum_{i \in \{+,-\}} \int_0^T \int_{\partial \Omega_i} \Bigl\{\frac{1}{\lambda_i}(1-2k_iu_{i})\ddot{u}_{i}p_{i}+\frac{1}{\varrho_i}\nabla u_{i} \cdot \nabla p_{i}  \nn  \\
&+b_i(1-\delta_i)\nabla \dot{u}_{i} \cdot \nabla p_{i} +b_i\delta_i |\nabla \dot{u}_{i}|^{q-1}\nabla \dot{u}_{i} \cdot \nabla p_{i}
-\frac{2k_i}{\lambda_i}(\dot{u}_{i})^2p_{i}\Bigr\}\, h^T n_i \, dx \, ds  \\
&-\displaystyle \sum_{i \in \{+,-\}}\int_0^T \int_{\partial \Omega_i}\Bigl\{\frac{1}{\varrho_i} \frac{\partial u_{i}}{\partial n_{i}} 
+b_i(1-\delta_i)\frac{\partial \dot{u}_{i}}{\partial n_{i}}
+b_i\delta_i|\nabla \dot{u}_{i}|^{q-1}\frac{\partial \dot{u}_{i}}{\partial n_{i}}\Bigr\} \, (\nabla p_{i}^Th) \, dx \, ds  \\
& - \displaystyle \sum_{i \in \{+,-\}} \int_0^T \int_{\partial \Omega_i}\Bigl\{\frac{1}{\varrho_i} \frac{\partial p_{i}}{\partial n_{i}}-b_i(1-\delta_i)\frac{\partial \dot{p}_{i}}{\partial n_{i}}
-b_i\delta_i \,(G_{u_i}(\nabla p_i) \cdot n_i)^{\Lcdot} \Bigr\} \, (\nabla u_{i}^Th) \, dx \, ds \\&-\displaystyle \sum_{i \in \{+,-\}} \int_0^T \int_{\Omega_i} j^{\prime}(u_i) (\nabla u_{i}^Th) \, dx \, ds \nn \\
&+R_1(p_i,p_{i,m})+R_2(p_i,p_{i,m})+R_3(p_i,p_{i,m})+R_4(p_i,p_{i,m}), \nn 
\end{align*}
(the second sum is written in terms of $p_i$ plus the error $R_4$) with 
\begin{align*}
R_4(p_i,p_{i,m})=& \,\displaystyle \sum_{i \in \{+,-\}}\int_0^T \int_{\partial \Omega_i}\Bigl\{\frac{1}{\varrho_i} \frac{\partial u_{i}}{\partial n_{i}} 
+b_i(1-\delta_i)\frac{\partial \dot{u}_{i}}{\partial n_{i}}
+b_i\delta_i|\nabla \dot{u}_{i}|^{q-1}\frac{\partial \dot{u}_{i}}{\partial n_{i}}\Bigr\} \, \nabla (p_{i}-p_{i,m})^Th \, dx \, ds.
\end{align*}
Next, we want to show that $$R(p_i,p_{i,m}):=R_1(p_i,p_{i,m})+R_2(p_i,p_{i,m})+R_3(p_i,p_{i,m})+R_4(p_i,p_{i,m})$$ tends to zero as $m \rightarrow \infty$. We will focus here on the estimates for the boundary integrals. \\

\noindent \noindent Since $p_i-p_{i,m}=0$ on $\Gamma$, we know that $\nabla_\Gamma (p_i-p_{i,m})=0$, where $\nabla_\Gamma$ denotes the tangential gradient. This further implies that 
\begin{align}                                                                                                                                                                                                                                                                                                                                                                                                                                                                                                                                                                                                                                                                                                                                                                                                                                                                                                                                                                                                                                                                                                                                                                                                                                                                  
&\nabla (p_i-p_{i,m})\vert_{\Gamma} \cdot h=\frac{\partial (p_i-p_{i,m})}{\partial n_i} (h \cdot n_i),\label{identity1}\\
&\nabla (p_i-p_{i,m})\vert_{\Gamma} \cdot \nabla u_i\vert_{\Gamma}=\frac{\partial (p_i-p_{i,m})}{\partial n_i} \frac{\partial u_i}{\partial n_i}.\label{identity2}
\end{align} 
\noindent Due to \eqref{identity1} and the fact that $h=0$ on $\partial \Omega_- \setminus \Gamma$, we can estimate
\begin{align*}
R_4(p_i,p_{i,m}) =& \, \displaystyle \sum_{i \in \{+,-\}}\int_0^T \int_{\partial \Omega_i}\Bigl\{\frac{1}{\varrho_i} \frac{\partial u_{i}}{\partial n_{i}} 
+b_i(1-\delta_i)\frac{\partial \dot{u}_{i}}{\partial n_{i}}\\
&+b_i\delta_i|\nabla \dot{u}_{i}|^{q-1}\frac{\partial \dot{u}_{i}}{\partial n_{i}}\Bigr\} \, \frac{\partial (p_{i}-p_{i,m})}{\partial n_i} h^Tn_i \, dx \, ds \\
\leq& \,|h|_{L^\infty(\Gamma)}\displaystyle \sum_{i \in \{+,-\}}\Bigl(\frac{1}{\varrho_i}\Bigl\|\frac{\partial u_{i}}{\partial n_{i}} \Bigr\|_{L^2(0,T;H^{1/2}(\Gamma))}+b_i(1-\delta_i)\Bigl\|\frac{\partial \dot{u}_{i}}{\partial n_{i}} \Bigr\|_{L^2(0,T;H^{1/2}(\Gamma))}\Bigr)\\
&+b_i\delta_i \|\nabla \dot{u}_i\|^{q-1}_{L^\infty(0,T;L^{\infty}(\Gamma))}\|\frac{\partial \dot{u}_{i}}{\partial n_{i}} \Bigr\|_{L^2(0,T;H^{1/2}(\Gamma))}\Bigl) \Bigl\|\frac{\partial (p_{i}-p_{i,m})}{\partial n_i} \Bigr\|_{L^2(0,T;H^{-1/2}(\Gamma))} \\
\leq& \,C|h|_{L^\infty(\Gamma)}\displaystyle \sum_{i \in \{+,-\}}\Bigl(\frac{1}{\varrho_i}\Bigl\|\frac{\partial u_{i}}{\partial n_{i}} \Bigr\|_{L^2(0,T;H^{1/2}(\Gamma))}+b_i(1-\delta_i)\Bigl\|\frac{\partial \dot{u}_{i}}{\partial n_{i}} \Bigr\|_{L^2(0,T;H^{1/2}(\Gamma))}\Bigr)\\
&+b_i\delta_i \|\nabla \dot{u}_i\|^{q-1}_{L^\infty(0,T;L^{\infty}(\Gamma))}\|\frac{\partial \dot{u}_{i}}{\partial n_{i}} \Bigr\|_{L^2(0,T;H^{1/2}(\Gamma))}\Bigl)\|p_{i}-p_{i,m}\|_{L^2(0,T;H^{1}(\Omega_i))} \\
\rightarrow& \, 0, \quad \text{as} \ m \rightarrow \infty.
\end{align*}
Here we have made use of the fact that since $u_i \in H^1(0,T;H^2(\Omega_i))$, we have $\frac{\partial u_i}{\partial n_i} \in H^1(0,T;H^{1/2}(\partial \Omega_i))$ (provided that $\Omega_i$ has a $C^{1,1}$ boundary, which we have assumed). Similarly, by employing \eqref{identity2}, we obtain
\begin{align*}
R_3(p_i,p_{i,m}) &=-\displaystyle \sum_{i \in \{+,-\}} \int_0^T \int_{\partial \Omega_i} \Bigl\{\frac{1}{\lambda_i}(1-2k_iu_{i})\ddot{u}_{i}(p_i-p_{i,m})+\frac{1}{\varrho_i}\frac{\partial u_{i}}{\partial n_{i}} \frac{\partial (p_{i}-p_{i,m})}{\partial n_i}   \nn  \\
&+b_i(1-\delta_i) \frac{\partial \dot{u}_{i}}{\partial n_{i}}  \frac{\partial (p_{i}-p_{i,m})}{\partial n_i}  +b_i\delta_i |\nabla \dot{u}_{i}|^{q-1}\frac{\partial \dot{u}_{i}}{\partial n_{i}} \frac{\partial (p_{i}-p_{i,m})}{\partial n_i} \\
&-\frac{2k_i}{\lambda_i}(\dot{u}_{i})^2(p_{i}-p_{i,m})\Bigr\}\, h^T n_i \, dx \, ds \\
\leq& |h|_{L^\infty(\Gamma)} \,\displaystyle \sum_{i \in \{+,-\}}\Bigl\{\frac{1}{\lambda_i}\|\ddot{u}_i\|_{L^2(0,T;L^2(\Gamma))}\|1-2k_iu_i\|_{L^\infty(0,T;L^{\infty}(\Gamma))}\|p_i-p_{i,m}\|_{L^2(0,T;L^2(\Gamma))}\\
&+\Bigl( \frac{1}{\varrho_i}\Bigl\|\frac{\partial u_{i}}{\partial n_{i}} \Bigr\|_{L^2(0,T;H^{1/2}(\Gamma))}+b_i(1-\delta_i)\Bigl\|\frac{\partial \dot{u}_{i}}{\partial n_{i}} \Bigr\|_{L^2(0,T;H^{1/2}(\Gamma))}\\
&+b_i\delta_i \|\nabla \dot{u}_i\|^{q-1}_{L^\infty(0,T;L^{\infty}(\Gamma))}\|\frac{\partial \dot{u}_{i}}{\partial n_{i}} \Bigr\|_{L^2(0,T;H^{1/2}(\Gamma))}\Bigr)\|p_{i}-p_{i,m}\|_{L^2(0,T;H^{1}(\Omega_i))}\\
&+\frac{2k_i}{\lambda_i}\|\dot{u}_i\|_{L^\infty(0,T;L^\infty(\Gamma))}\|\dot{u}_i\|_{L^2(0,T;L^2(\Gamma))}\|p_i-p_{i,m}\|_{L^2(0,T;L^2(\Gamma))} \Bigr\} \\
\rightarrow& \, 0, \quad \text{as} \ m \rightarrow \infty.
\end{align*}
Altogether, this means that
\begin{align*}
&\displaystyle \lim_{\tau \rightarrow 0} \frac{1}{\tau} \, \langle \tilde{E}(u,\tau) -\tilde{E}(u,0),p\rangle_{\tilde{X}^\star,\tilde{X}}  \nn\\
=& \, \displaystyle \sum_{i \in \{+,-\}} \int_0^T \int_{\partial \Omega_i} \Bigl\{\frac{1}{\lambda_i}(1-2k_iu_{i})\ddot{u}_{i}p_{i}+\frac{1}{\varrho_i}\nabla u_{i} \cdot \nabla p_{i}  \nn  \\
&+b_i(1-\delta_i)\nabla \dot{u}_{i} \cdot \nabla p_{i} +b_i\delta_i |\nabla \dot{u}_{i}|^{q-1}\nabla \dot{u}_{i} \cdot \nabla p_{i}
-\frac{2k_i}{\lambda_i}(\dot{u}_{i})^2p_{i}\Bigr\}\, h^T n_i \, dx \, ds  \\
&-\displaystyle \sum_{i \in \{+,-\}}\int_0^T \int_{\partial \Omega_i}\Bigl\{\frac{1}{\varrho_i} \frac{\partial u_{i}}{\partial n_{i}} 
+b_i(1-\delta_i)\frac{\partial \dot{u}_{i}}{\partial n_{i}}
+b_i\delta_i|\nabla \dot{u}_{i}|^{q-1}\frac{\partial \dot{u}_{i}}{\partial n_{i}}\Bigr\} \, (\nabla p_{i}^Th) \, dx \, ds  \\
& - \displaystyle \sum_{i \in \{+,-\}} \int_0^T \int_{\partial \Omega_i}\Bigl\{\frac{1}{\varrho_i} \frac{\partial p_{i}}{\partial n_{i}}-b_i(1-\delta_i)\frac{\partial \dot{p}_{i}}{\partial n_{i}}
-b_i\delta_i \,(G_{u_i}(\nabla p_i) \cdot n_i)^{\Lcdot} \Bigr\} \, (\nabla u_{i}^Th) \, dx \, ds \\&-\displaystyle \sum_{i \in \{+,-\}} \int_0^T \int_{\Omega_i} j^{\prime}(u_i) (\nabla u_{i}^Th) \, dx \, ds. \nn
\end{align*}
\noindent Assume that $u_d \in L^2(0,T;H^1_0(\Omega))$. We can utilize the Stokes theorem and the fact that $h=0$ on $\partial \Omega_- \setminus \Gamma$ to acquire
\begin{align*}
&\displaystyle \sum_{i \in \{+,-\}}  \int_0^T \int_{\Omega_i}j(u_i) \, \div h \, dx \, ds+ \displaystyle \sum_{i \in \{+,-\}} \int_0^T \int_{\Omega_i}j^\prime (u_i) (\nabla u_i^T h) \, dx \, ds \\
=&\,  \displaystyle \sum_{i \in \{+,-\}}\int_0^T \int_{\Gamma} j(u_i) \, h^T n_i \, dx \, ds,
\end{align*}
which leads to the shape derivative given in terms of the boundary integrals
\begin{align*}
dJ(u,\Omega_+)h
=& \,- \int_0^T \int_{\Gamma} \Bigl \llbracket \frac{1}{\lambda}(1-2ku)\ddot{u}p+\frac{1}{\varrho}\nabla u \cdot \nabla p+b(1-\delta)\nabla \dot{u} \cdot \nabla p \\
&+b\delta |\nabla \dot{u}|^{q-1}\nabla \dot{u} \cdot \nabla p-\frac{2k}{\lambda}(\dot{u})^2p\Bigr \rrbracket \, h^Tn_+ \, dx \, ds \\
&+\int_0^T \int_{\Gamma} \Bigl \llbracket (\frac{1}{\varrho}\nabla u+b(1-\delta)\nabla \dot{u}+b\delta |\nabla \dot{u}|^{q-1}\nabla \dot{u})\cdot n_+ (\nabla p \cdot h) \Bigr \rrbracket \, dx \, ds \\
&+\int_0^T \int_{\Gamma} \Bigl \llbracket \Bigl(\frac{1}{\varrho}\nabla p-b(1-\delta)\nabla \dot{p}-b\delta(G_u(\nabla p))^{\Lcdot} \Bigr)\cdot n_+ (\nabla u \cdot h)\Bigr \rrbracket \, dx \, ds.
\end{align*}
\noindent Here we have used that $\llbracket (u-u_d)^2 \rrbracket=0$. The expression for the shape derivative can be slightly simplified. For the second and third integral on the right hand side, by employing the fact that (see Example 2, \cite{IKP08})
\begin{align*}
\llbracket x \rrbracket=\llbracket y \rrbracket=0 \implies \llbracket xy \rrbracket=0,
\end{align*} 
we obtain the following identities
\begin{align*}
&\Bigl \llbracket(\frac{1}{\varrho}\nabla u+b(1-\delta)\nabla \dot{u}+b\delta |\nabla \dot{u}|^{q-1}\nabla \dot{u})\cdot n_+ (\nabla p \cdot h) \Bigr \rrbracket  \\
=& \,\Bigl \llbracket(\frac{1}{\varrho}\frac{\partial u}{\partial n_+}\frac{\partial  p}{\partial n_+}+b(1-\delta)\frac{\partial  \dot{u}}{\partial n_+}\frac{\partial  p}{\partial n_+}+b\delta |\nabla \dot{u}|^{q-1}\frac{\partial  \dot{u}}{\partial n_+}\frac{\partial  p}{\partial n_+})(h \cdot n_+) \\
&+\Bigl(\frac{1}{\varrho}\nabla u+b(1-\delta)\nabla \dot{u}+b\delta |\nabla \dot{u}|^{q-1}\nabla \dot{u}\Bigr)\cdot n_+\nabla_{\Gamma}p \cdot h\Bigr \rrbracket\\
=& \,\Bigl \llbracket \frac{1}{\varrho}\frac{\partial  u}{\partial n_+}\frac{\partial  p}{\partial n_+}+b(1-\delta)\frac{\partial  \dot{u}}{\partial n_+}\frac{\partial  p}{\partial n_+}+b\delta |\nabla \dot{u}|^{q-1}\frac{\partial  \dot{u}}{\partial n_+}\frac{\partial  p}{\partial n_+}\Bigr \rrbracket\,(h \cdot n_+).
\end{align*}
and similarly
\begin{align*}
&\int_0^T \int_{\Gamma}\Bigl \llbracket \Bigl(\frac{1}{\varrho}\nabla p-b(1-\delta)\nabla \dot{p}-b\delta(G_u(\nabla p))^{\Lcdot} \Bigr)\cdot n_+ (\nabla u \cdot h)\Bigr \rrbracket\, dx \, ds\\
=&\,\int_0^T \int_{\Gamma} \Bigl \llbracket \frac{1}{\varrho}\frac{\partial  u}{\partial n_+}\frac{\partial p}{\partial n_+}+b(1-\delta)\frac{\partial  \dot{u}}{\partial n_+}\frac{\partial p}{\partial n_+}+b\delta \,|\nabla \dot{u}|^{q-1}\frac{\partial \dot{u}}{\partial n_+}\frac{\partial p}{\partial n_+} \\
&+b\delta (q-1)|\nabla \dot{u}|^{q-3} (\nabla \dot{u} \cdot \nabla p)\Bigl|\frac{\partial  \dot{u}}{\partial n_+}\Bigr|^2\Bigr \rrbracket \,(h \cdot n_+) \, dx \, ds. 
\end{align*}
Here we have made use of the fact that $\llbracket \nabla_{\Gamma} u \rrbracket=\llbracket \nabla_{\Gamma} p \rrbracket=0$. We finally obtain
\begin{theorem}(Strong shape derivative) 
 Let $\partial \Omega$ and $\Gamma=\partial \Omega_+$ be $C^{1,1}$, $u_0\vert_{\Omega_i} \in H^2(\Omega_i)$, $u_0,u_1 \in W_0^{1,q+1}(\Omega)$, $q >2$, and let assumptions \eqref{coeff_3} on the coefficients in the state equation and hypotheses \text{\normalfont ($\mathcal{H}_1$)-($\mathcal{H}_5$)} hold true. Assume that $u_d \in L^2(0,T;H^1_0(\Omega))$. The shape derivative of $J$ at $\Omega_+$ in the direction of a vector field $h \in C^{1,1}(\bar{\Omega},\mathbb{R}^d)$ is given by
\begin{equation} \label{strong_shaped_long}
\begin{aligned}
dJ(u,\Omega_+)h
=& \, \int_0^T \int_{\Gamma} \Bigl \llbracket-\frac{1}{\lambda}(1-2ku)\ddot{u}p-\frac{1}{\varrho}\nabla u \cdot \nabla p\\
&-b((1-\delta)+\delta |\nabla \dot{u}|^{q-1})\nabla \dot{u} \cdot \nabla p+\frac{2k}{\lambda}(\dot{u})^2p\\
&+\frac{2}{\varrho}\frac{\partial  u}{\partial n_+}\frac{\partial p}{\partial n_+}+2b((1-\delta)+\delta \,|\nabla \dot{u}|^{q-1})\frac{\partial  \dot{u}}{\partial n_+}\frac{\partial p}{\partial n_+} \\
&+b\delta (q-1)|\nabla \dot{u}|^{q-3} (\nabla \dot{u} \cdot \nabla p)\Bigl|\frac{\partial  \dot{u}}{\partial n_+}\Bigr|^2\Bigr \rrbracket\, h^Tn_+ \, dx \, ds.
\end{aligned} 
\end{equation}
\end{theorem}
The boundary integrals in \eqref{strong_shaped_long} are well-defined thanks to hypotheses $(\mathcal{H}_4)$ and $(\mathcal{H}_5)$.
\section{Conclusion}
We have computed, through a variational approach, the weak and the strong shape derivative for the cost functional determining the acoustic pressure of high intensity ultrasound when focusing is performed by an acoustic lens. \\
\indent Future research will be directed at developing and implementing a suitable gradient based optimization algorithm, as well as considering a physically more involved elastic model for the focusing lens and thus a shape optimization problem with an elastic-acoutic coupling as the optimization constraint.
\subsection*{Acknowledgments} The financial support by the FWF (Austrian Science Fund) under grant P24970 is gratefully acknowledged as well as the support of the Karl Popper Kolleg "Modeling-Simulation-Optimization", which is funded by the Alpen-Adria-Universit\" at Klagenfurt and by the Carinthian Economic Promotion Fund (KWF).

\end{document}